\documentclass[12pt]{article}

\usepackage{amsmath,amsthm,amsfonts,amssymb,amscd,cite}
\usepackage{graphicx,subfigure}
\usepackage[title]{appendix}
\allowdisplaybreaks[4]

\newtheorem {lemma}{Lemma}[section]
\newtheorem {theorem} {Theorem}[section]

\newtheorem {corollary}{Corollary}[section]
\newtheorem {claim}{Claim}[section]

\newtheorem {question}{Question}[section]
\newtheorem {case}{Case}
\usepackage[cm]{fullpage}
\usepackage{tikz}
\usetikzlibrary{positioning}
\usepackage{xcolor}
\usepackage{circuitikz}
\usetikzlibrary{patterns}
\usetikzlibrary{shapes.geometric}

\begin{document}

\title{Signless Laplacian index conditions for  trebly chorded cycles in graphs with given order}

\author{Jin Cai\footnote{Email: jincai@m.scnu.edu.cn}, Bo Zhou\footnote{Email: zhoubo@m.scnu.edu.cn}\\
School of Mathematical Sciences, South China Normal University\\
Guangzhou 510631, P.R. China}

\date{}
\maketitle

\begin{abstract}
It is proved that for a  graph of order $n$, where $n\ge 6$, if the signless Laplacian index is  larger than or equal to certain value depending on $n$, then the graph contains a trebly chorded cycle, where the  chords incident to a common vertex, unless it is one of two specified graphs.
\end{abstract}

\section{Introduction}
Let $G$ be a  simple graph with vertex set $V(G)$ and edge set $E(G)$.
A path on $n$ vertices is denoted  $P_n$.
For vertex disjoint graph $G$ and $H$, $G\cup H$ denotes their union, which is the graph
with vertex set $V(G)\cup V(H)$ and edge set $E(G)\cup E(H)$,
and
$G\vee H$ denotes their join, which  is the graph obtained from $G\cup H$ by adding all possible edges between vertices in $G$ and vertices in $H$.

A chord of a cycle
$C$ is an edge joining two non-consecutive vertices of $C$.
A cycle $C$ is chorded if there is a chord,  doubly chorded if there are two chords, and
trebly chorded if there are three chords.
Answering P\'{o}sa's
question  \cite{Po}, 
Czipzer \cite{Cz} proved
that any graph with minimum degree at least three contains a chorded cycle  (see Lov\'{a}sz  \cite{Lov}, problem 10.2 and its solution).
Some researchers found  conditions that imply the
existence of doubly chorded cycles \cite{QZ, CFG,SV,GHH,W97,GLW}. In  the recent survey \cite{Gou},
Gould asked the question: What spectral conditions imply a graph contains a chorded cycle?
Answers are given for graphs with fixed order via index (spectral radius) \cite{ZHW} and signless Laplacian index (signless Laplacian spectral radius) \cite{XZ}, respectively.
Wang and Zhai \cite{WZ} found the signless Laplacian index conditions that imply the existence of $P_1\vee P_{2k}$, so they actually gave the signless Laplacian index conditions that imply the existence of cycles with even number of chords  incident to a vertex.
As far as we know, there is no tight spectral condition that  implies the existence of trebly chorded cycles.
A particular case of Jiang \cite{Ji} states that for any $n$-vertex graph with $n\ge 12$, if $|E(G)|\ge 4n-16$, then $G$ contains  a trebly chorded cycle  with three chords incident to a vertex.
Inspired by the above works,
we study the following question:

\begin{question} \label{Qu}
What tight signless Laplacian index conditions imply a graph with fixed order contains a trebly chorded cycle (with chords incident to a vertex)?
\end{question}

For a graph $G$, we denote by $q(G)$ its signless Laplacian index, which is defined to be the largest eigenvalue of the signless Laplacian matrix of $G$.

Denote by $K_n$ the complete graph on $n$ vertices and $K_{n_1,\dots, n_k}$ the complete $k$-partite graph with classes containing $n_1,\dots, n_k$ vertices, respectively.
Denote by $K_{1,1,n-2}^+$ the graph obtained from $K_{1,1,n-2}$ by adding one edge in the part with size $n-2$, where $n\ge 4$.

For $n=5$, it is easy to find that $K_5$ does not contain a trebly chorded cycle with three chords incident to a vertex. The main result is listed as below.

\begin{theorem}\label{doubly21}
Suppose that $G$ is a graph of order $n$ without isolated vertices, where $n\ge 6$.
If $n=6$ and $q(G)\ge q(K_1\vee (K_4\cup K_1))$, then $G$ contains a trebly chorded cycle with three chords incident to a vertex unless $G\cong K_1\vee (K_4\cup K_1)$.
If $n\ge 7$ and $q(G)\ge q(K_{1,1,n-2}^+)$, then $G$ contains a trebly chorded cycle with three chords incident to a vertex unless $G\cong K_{1,1,n-2}^+$. 
\end{theorem}

%

\section{Preliminaries}

Let $G$ be a graph.
For $v\in V(G)$, we denote by $N_G(v)$ the neighborhood of $v$ in $G$, and $d_G(v)$ the degree of $v$ in $G$. Denote by $\Delta(G)$ and $\delta(G)$ the maximum degree and minimum degree of $G$, respectively. Let $N_G[v]=\{v\}\cup N_G(v)$.  A pendant vertex is a vertex of degree one.
For a graph $G$ with $\emptyset\ne S\subseteq V(G)$, denote by $G[S]$ the subgraph of $G$ induced by $S$.
For simplicity, for a fixed graph $G$ with $v\in V(G)$, we write $N(v)$ for $N_G(v)$ and $d(v)$ for $d_G(v)$, and we denote $N_{G[S]}(v)$ by $N_S(v)$ and frequently write $d_{G[S]}(v)$ as $d_S(v)$  when $S\subset V(G)$. For $S,T\subset V(G)$, let $e(S,T)$ be the number of edges with one end in $S$ and one end in $T$. Particularly, if $S=T$, then $e(S)$ denotes the number of edges in $G[S]$.
Denote by $C_n$ the cycle on $n$ vertices, where $n\ge 3$.
A star means $K_1$ or some $K_{1,s}$ with $s\ge 1$, in which  the vertex of degree $s$ when $s\ge 2$ (any vertex when $s=1$) is the center and a vertex of degree one is a leaf. The vertex in $K_1$ is its center.
Denote by $S_{n_1,n_2}$ the double star with order $n_1+n_2+2$ which is obtained by adding an edge between the centers of two disjoint star $K_{1,n_1}$ and $K_{1,n_2}$, say the centers of $K_{1,n_1}$ and $K_{1,n_2}$ are the centers of the double star.
Denote $K_{1,s}^+$ with $s\ge 2$ the graph obtained from $K_{1,s}$ by adding edge $e$ between two non-adjacent vertices of $K_{1,s}$. The center of $K_{1,s}$ is the center of $K_{1,s}^+$.
Denote $C_4^+$ the graph obtained from $C_4$ by adding edge $e$ between two non-adjacent vertices of $C_4$.
The disjoint union of $k$ copies of a graph $G$ is denoted by $kG$. For graphs $H_1$ and $H_2$, $H_1\subseteq H_2$ means that $H_1$ is a subgraph of $H_2$.


For a graph $G$ with $E_1\subseteq E(G)$, denote by $G-E_1$ the graph with vertex set $V(G)$ and edge set
$E(G)\setminus E_1$, and in particular, we write $G-f$ for $G-\{f\}$ for $f\in E(G)$. If $E_2$ is a subset of the edge set of the complement of $G$, then $G+E_2$ denotes the graph with vertex set $V(G)$ and edge set $E(G)\cup E_2$, and in particular, we write $G+f$ for $G+\{f\}$ if $E_2=\{f\}$.

For an $n$-vertex graph $G$, the adjacency matrix of $G$ is the $n\times n$ matrix
$A(G)=(a_{uv})_{u,v\in V(G)}$, where $a_{uv}=1$ if $uv\in E(G)$ and $a_{uv}=0$ otherwise.
Let $D(G)$ be the degree diagonal matrix of $G$.
The matrix $Q(G)=A(G)+D(G)$ is known as the signless Laplacian matrix of $G$.

The following lemma is an immediate consequence of the Perron-Frobenius theorem.

\begin{lemma}\label{addedges}
Let $G$ be a graph and $H$ be a subgraph of $G$. Then $q(H)\le q(G)$. Moreover, if $H$ is a proper subgraph of $G$ and $G$ is connected,  then $q(H)<q(G)$.
\end{lemma}

If $G$ is a connected graph, then $Q(G)$ is irreducible and so by Perron-Frobenius theorem, there exists a unique unit positive eigenvector of $Q(G)$ corresponding to $q(G)$, which we call the Perron vector of $Q(G)$.

The following lemma is a special case of Lemma 6 in \cite{NR}, see also \cite{Row,SLB}.

\begin{lemma} \label{perron}
Let $G$ be a graph with $\{u,v\}\subset V(G)$ and $\emptyset\ne S\subseteq N_G(v)\setminus N_G[u]$. Let
$G'=G-\{vw: w\in S\}+\{uw: w\in S\}$. If $\mathbf{x}$ is the Perron vector of $G$ with $x_u\ge x_v$, then $q(G')>q(G)$.
\end{lemma}


For a nonnegative square matrix $M$, denote by $\lambda(M)$ its spectral radius.

Suppose that $V(G)$ is partitioned as $V_1\cup \dots\cup V_m$. For $1\le i\le j\le m$, set $Q_{ij}$ to be the submatrix of $Q(G)$ with rows corresponding to vertices in $V_i$ and columns corresponding to vertices in $V_j$. The quotient matrix of $Q(G)$ with respect to the partition $V_1\cup \dots \cup V_m$ the matrix  $B=(b_{ij})$, where $b_{ij}=\frac{1}{|V_i|}\sum_{u\in V_i}\sum_{v\in V_j}q_{uv}$. If $Q_{ij}$ has constant row sum for  $1\le i\le j\le m$, then we say $B$ is an equitable quotient matrix of $Q(G)$.
The following lemma is an immediate consequence of \cite[Lemma 2.3.1]{BH}.

\begin{lemma}\label{quo}
For a connected graph $G$, if $B$ is an equitable quotient matrix of $Q(G)$, then $\lambda(B)=q(G)$.
\end{lemma}

For a graph $G$ and $v\in V(G)$ with $d(v)>0$, let $\eta_G(v)=d(v)+\frac{1}{d(v)}\sum_{uv\in E(G)}d(u)$.
Note that

\begin{align}
\eta_{G}(v)&=d(v)+\frac{d(v)+2e(G[N(v)])+e(N(v),V(G)\setminus N[v])}{d(v)}\label{q11}\\
&\le  d(v)+1+\frac{2e(G[N(v)])+d(v)|V(G)\setminus N[v]|}{d(v)}=n+\frac{2e(G[N(v)])}{d(v)}.\label{q12}
\end{align}

\begin{lemma}\label{q1}\cite{FY}
For any graph $G$ with no isolated vertices, we have
\[
q(G)\le \max \{\eta_G(v): v\in V(G)\}.
\]
Moreover, if $G$ is connected, then the equality holds if and only if $G$ is a regular or semi-regular bipartite graph.
\end{lemma}

\begin{lemma}\label{e1}\cite{JAB}
Let $C$ be a longest cycle with length $c$ of a graph $G$ on $n$ vertices. Then there are at most $\frac{1}{2}c\left(n-c\right)$ edges of $G$ with at most one endpoint on $C$.
\end{lemma}

\begin{lemma}\label{ee}
Let $G$ be a graph of order $n\ge 10$. If $G$ does not contain a trebly chorded cycle with three chords incident to a vertex, then $e(G)\le 4n-16$.
\end{lemma}
\begin{proof}
Let $C$ be a longest cycle with length $c$ of $G$. From Lemma \ref{e1},
\[
e(G-E(G[V(C)]))\le \frac{1}{2}c\left(n-c\right).
\]
As $G$ does not contain a trebly chorded cycle with three chords incident to a vertex, each vertex of $G[V(C)]$ has degree at most four.  Then $e(G[V(C)])\le 2c$. So
\[
e(G)= e(G-E(G[V(C)]))+e(G[V(C)])\le \frac{1}{2}c\left(n-c\right)+2c=-\frac{1}{2}c^2+\left(\frac{1}{2}n+2\right)c.
\]
If $n=10$, then
\[
e(G)\le -\frac{1}{2}c^2+7c\le -\frac{1}{2}\times 7^2+7\times 7=\frac{49}{2},
\]
So $e(G)\le 24=4\times 10-16$, as desired.
If $n=11$, then
\[
e(G)\le -\frac{1}{2}c^2+\frac{15}{2}c\le -\frac{1}{2}\times 7^2+\frac{15}{2}\times 7=28=4\times 11-16,
\]
as desired.
Suppose that $n\ge 12$, then $e(G)\le 4n-16$ by Jiang's result  \cite{Ji}.
\end{proof}

\begin{lemma}\label{p5}\cite{PE}
Let $k \ge 1$. If $G$ is a graph of order $n$ with no $P_{k+2}$, then $e(G)\le \frac{nk}{2}$, with equality holding if and only if $G$ is a union of disjoint copies of $K_{k+1}$.
\end{lemma}

In order to prove our main result, we will use the following Lemma.
\begin{lemma}\label{712}
For $n\ge 6$, we have $q(K_{1,1,n-2}^+)>n+2-\frac{4}{n+2}$.
\end{lemma}
\begin{proof}
Since $Q(K_{1,1,n-2}^+)$ has an equitable quotient matrix
$
B=
\left(\begin{smallmatrix}
n &  2  & n-4  \\
2 & 4  &  0  \\
2 &  0  &  2
\end{smallmatrix}\right)$
with CP (characteristic polynomial)
\[
g(x)=x^3-\left(n+6\right)x^2+\left(4n+12\right)x-24.
\]
As $g\left(n+2-\frac{4}{n+2}\right)=-\frac{8n^3+32n^2+96n+192}{n^3+6n^2+12n+8}<0$,
we have $q(K_{1,1,n-2}^+)>n+2-\frac{4}{n+2}$ from Lemma \ref{quo}.
\end{proof}

\begin{figure}[htbp]
\centering
\begin{tabular}{cccc}
\begin{tikzpicture}[scale=0.6]
\filldraw [black] (1,1.732) circle (2pt);
\filldraw [black] (1,-1.732) circle (2pt);
\filldraw [black] (-1,-1.732) circle (2pt);
\filldraw [black] (-1,1.732) circle (2pt);
\draw [black] (1,1.732)--(1,-1.732)--(-1,-1.732)--(-1,1.732)--(1,1.732);
\filldraw [black] (-2,0) circle (2pt);
\node at (-2.3,0) {$z$};
\filldraw [black] (2,0) circle (2pt);
\node at (2.4,0) {$w$};
\draw [black] (-2,0)--(1,1.732);
\draw [black] (-2,0)--(1,-1.732);
\draw [black] (-2,0)--(-1,-1.732);
\draw [black] (-2,0)--(-1,1.732);
\draw [black] (2,0)--(1,1.732);
\draw [black] (2,0)--(1,-1.732);
\draw [black] (2,0)--(-1,-1.732);
\draw [black] (2,0)--(-1,1.732);
\end{tikzpicture}&
\begin{tikzpicture}[scale=0.6]
\filldraw [black] (1,1.732) circle (2pt);
\filldraw [black] (1,-1.732) circle (2pt);
\filldraw [black] (-1,-1.732) circle (2pt);
\filldraw [black] (-1,1.732) circle (2pt);
\draw [black] (1,1.732)--(1,-1.732)--(-1,-1.732)--(-1,1.732)--(1,1.732);
\draw [black] (1,1.732)--(-1,-1.732);
\filldraw [black] (-2,0) circle (2pt);
\node at (-2.3,0) {$z$};
\filldraw [black] (2,0) circle (2pt);
\node at (2.4,0) {$w$};
\draw [black] (-2,0)--(1,1.732);
\draw [black] (-2,0)--(1,-1.732);
\draw [black] (-2,0)--(-1,-1.732);
\draw [black] (-2,0)--(-1,1.732);
\draw [black] (2,0)--(1,-1.732);
\draw [black] (2,0)--(-1,1.732);
\end{tikzpicture}
&
\begin{tikzpicture}[scale=0.6]
\filldraw [black] (1,1.732) circle (2pt);
\filldraw [black] (1,-1.732) circle (2pt);
\filldraw [black] (-1,-1.732) circle (2pt);
\filldraw [black] (-1,1.732) circle (2pt);
\draw [black] (1,1.732)--(1,-1.732)--(-1,-1.732)--(-1,1.732)--(1,1.732);
\filldraw [black] (-2,0) circle (2pt);
\node at (-2.3,0) {$z$};
\filldraw [black] (2,0) circle (2pt);
\node at (2.4,0) {$w$};
\filldraw [black] (-2,1.732) circle (2pt);
\draw [black] (-2,0)--(1,1.732);
\draw [black] (-2,0)--(1,-1.732);
\draw [black] (-2,0)--(-1,-1.732);
\draw [black] (-2,0)--(-1,1.732);
\draw [black] (-2,0)--(-2,1.732);
\draw [black] (2,0)--(1,1.732);
\draw [black] (2,0)--(1,-1.732);
\draw [black] (2,0)--(-1,-1.732);
\draw [black] (2,0)--(-1,1.732);
\end{tikzpicture}
&
\begin{tikzpicture}[scale=0.6]
\filldraw [black] (1,1.732) circle (2pt);
\filldraw [black] (1,-1.732) circle (2pt);
\filldraw [black] (-1,-1.732) circle (2pt);
\filldraw [black] (-1,1.732) circle (2pt);
\draw [black] (1,1.732)--(1,-1.732)--(-1,-1.732)--(-1,1.732)--(1,1.732);
\draw [black] (1,1.732)--(-1,-1.732);
\filldraw [black] (-2,0) circle (2pt);
\node at (-2.3,0) {$z$};
\filldraw [black] (2,0) circle (2pt);
\node at (2.4,0) {$w$};
\filldraw [black] (-2,1.732) circle (2pt);
\draw [black] (-2,0)--(1,1.732);
\draw [black] (-2,0)--(1,-1.732);
\draw [black] (-2,0)--(-1,-1.732);
\draw [black] (-2,0)--(-1,1.732);
\draw [black] (-2,0)--(-2,1.732);
\draw [black] (2,0)--(1,-1.732);
\draw [black] (2,0)--(-1,1.732);
\end{tikzpicture}\\
$U_1$ & $U_2$ & $U_3$ & $U_4$\\[2mm]
\begin{tikzpicture}[scale=0.6]
\filldraw [black] (1.375,1.732) circle (2pt);
\filldraw [black] (1.375,-1.732) circle (2pt);
\filldraw [black] (0.125,1.732) circle (2pt);
\filldraw [black] (0.125,-1.732) circle (2pt);
\draw [black] (1.375,1.732)--(1.375,-1.732)--(0.125,-1.732)--(0.125,1.732)--(1.375,1.732);
\filldraw [black] (-0.5,0) circle (2pt);
\node at (-0.6,-0.4) {$z$};
\filldraw [black] (2,0) circle (2pt);
\node at (2.4,0) {$w$};
\filldraw [black] (-2,-1.732) circle (2pt);
\filldraw [black] (-1.25,0) circle (2pt);
\filldraw [black] (-2,1.732) circle (2pt);
\draw [black] (-2,-1.732)--(-1.25,0)--(-2,1.732)--(-2,-1.732);
\draw [black] (-0.5,0)--(1.375,1.732);
\draw [black] (-0.5,0)--(1.375,-1.732);
\draw [black] (-0.5,0)--(0.125,1.732);
\draw [black] (-0.5,0)--(0.125,-1.732);
\draw [black] (-0.5,0)--(-2,-1.732);
\draw [black] (-0.5,0)--(-1.25,0);
\draw [black] (-0.5,0)--(-2,1.732);
\draw [black] (2,0)--(1.375,1.732);
\draw [black] (2,0)--(1.375,-1.732);
\draw [black] (2,0)--(0.125,1.732);
\draw [black] (2,0)--(0.125,-1.732);
\end{tikzpicture}
&
\begin{tikzpicture}[scale=0.6]
\filldraw [black] (1.375,1.732) circle (2pt);
\filldraw [black] (1.375,-1.732) circle (2pt);
\filldraw [black] (0.125,1.732) circle (2pt);
\filldraw [black] (0.125,-1.732) circle (2pt);
\draw [black] (1.375,1.732)--(1.375,-1.732)--(0.125,-1.732)--(0.125,1.732)--(1.375,1.732);
\draw [black] (1.375,1.732)--(0.125,-1.732);
\filldraw [black] (-0.5,0) circle (2pt);
\node at (-0.6,-0.4) {$z$};
\filldraw [black] (2,0) circle (2pt);
\node at (2.4,0) {$w$};
\filldraw [black] (-2,-1.732) circle (2pt);
\filldraw [black] (-1.25,0) circle (2pt);
\filldraw [black] (-2,1.732) circle (2pt);
\draw [black] (-2,-1.732)--(-1.25,0)--(-2,1.732)--(-2,-1.732);
\draw [black] (-0.5,0)--(1.375,1.732);
\draw [black] (-0.5,0)--(1.375,-1.732);
\draw [black] (-0.5,0)--(0.125,1.732);
\draw [black] (-0.5,0)--(0.125,-1.732);
\draw [black] (-0.5,0)--(-2,-1.732);
\draw [black] (-0.5,0)--(-1.25,0);
\draw [black] (-0.5,0)--(-2,1.732);
\draw [black] (2,0)--(1.375,-1.732);
\draw [black] (2,0)--(0.125,1.732);
\end{tikzpicture}&
\begin{tikzpicture}[scale=0.6]
\filldraw [black] (1,1.732) circle (2pt);
\filldraw [black] (2,1.732) circle (2pt);
\filldraw [black] (2,-1.732) circle (2pt);
\node at (2.4,-1.732) {$w$};
\filldraw [black] (1,-1.732) circle (2pt);
\draw [black] (1,1.732)--(2,1.732)--(2,-1.732)--(1,-1.732)--(1,1.732);
\draw [black] (1,1.732)--(2,-1.732);
\draw [black] (2,1.732)--(1,-1.732);
\filldraw [black] (0,0) circle (2pt);
\node at (0,-0.37) {$z$};
\filldraw [black] (-2,1.732) circle (2pt);
\filldraw [black] (-2,-1.732) circle (2pt);
\filldraw [black] (-1,0) circle (2pt);
\draw [black] (-2,1.732)--(-2,-1.732)--(-1,0)--(-2,1.732);
\draw [black] (0,0)--(-2,1.732);
\draw [black] (0,0)--(-2,-1.732);
\draw [black] (0,0)--(-1,0);
\draw [black] (0,0)--(1,1.732);
\draw [black] (0,0)--(2,1.732);
\draw [black] (0,0)--(1,-1.732);
\end{tikzpicture}&
\begin{tikzpicture}[scale=0.6]
\filldraw [black] (0,0) circle (2pt);
\filldraw [black] (2,0) circle (2pt);
\node at (2.4,0) {$w$};
\filldraw [black] (2,-1.732) circle (2pt);
\filldraw [black] (0,-1.732) circle (2pt);
\draw [black] (0,0)--(2,0)--(2,-1.732)--(0,-1.732)--(0,0);
\draw [black] (0,0)--(2,-1.732);
\draw [black] (2,0)--(0,-1.732);
\filldraw [black] (0,1.732) circle (2pt);
\filldraw [black] (-2,0) circle (2pt);
\node at (-2.3,0) {$z$};
\draw [black] (-2,0)--(0,1.732);
\draw [black] (-2,0)--(0,0);
\draw [black] (-2,0)--(0,-1.732);
\draw [black] (-2,0)--(2,-1.732);
\draw [black] (2,0)--(0,1.732);
\end{tikzpicture}\\
$U_5$ & $U_6$ & $U_7$ & $U_8$\\[2mm]
\begin{tikzpicture}[scale=0.6]
\filldraw [black] (1,1.732) circle (2pt);
\filldraw [black] (1,-1.732) circle (2pt);
\filldraw [black] (-1,-1.732) circle (2pt);
\filldraw [black] (-1,1.732) circle (2pt);
\draw [black] (1,1.732)--(-1,-1.732)--(-1,1.732)--(1,1.732);
\draw [black] (-1,1.732)--(1,-1.732);
\filldraw [black] (-2,0) circle (2pt);
\node at (-2.3,0) {$z$};
\filldraw [black] (2,0) circle (2pt);
\node at (2.4,0) {$w$};
\draw [black] (-2,0)--(1,1.732);
\draw [black] (-2,0)--(1,-1.732);
\draw [black] (-2,0)--(-1,-1.732);
\draw [black] (-2,0)--(-1,1.732);
\draw [black] (2,0)--(1,1.732);
\draw [black] (2,0)--(1,-1.732);
\draw [black] (2,0)--(-1,-1.732);
\end{tikzpicture}&
\begin{tikzpicture}[scale=0.6]
\filldraw [black] (1,1.732) circle (2pt);
\filldraw [black] (1,-1.732) circle (2pt);
\filldraw [black] (-1,-1.732) circle (2pt);
\filldraw [black] (-1,1.732) circle (2pt);
\filldraw [black] (1,0) circle (2pt);
\draw [black] (1,1.732)--(-1,-1.732)--(-1,1.732)--(1,1.732);
\draw [black] (-1,1.732)--(1,-1.732);
\draw [black] (-1,1.732)--(1,0);
\filldraw [black] (-2,0) circle (2pt);
\node at (-2.3,0) {$z$};
\filldraw [black] (2,0) circle (2pt);
\node at (2.4,0) {$w$};
\draw [black] (-2,0)--(1,1.732);
\draw [black] (-2,0)--(1,-1.732);
\draw [black] (-2,0)--(-1,-1.732);
\draw [black] (-2,0)--(-1,1.732);
\draw [black] (-2,0)--(1,0);
\draw [black] (2,0)--(1,1.732);
\draw [black] (2,0)--(1,-1.732);
\draw [black] (2,0)--(-1,-1.732);
\end{tikzpicture}&
\begin{tikzpicture}[scale=0.6]
\filldraw [black] (1,1.732) circle (2pt);
\filldraw [black] (1,-1.732) circle (2pt);
\filldraw [black] (-1,-1.732) circle (2pt);
\filldraw [black] (-1,1.732) circle (2pt);
\filldraw [black] (1,0.4) circle (2pt);
\filldraw [black] (1,-0.4) circle (2pt);
\draw [black] (1,1.732)--(-1,-1.732)--(-1,1.732)--(1,1.732);
\draw [black] (-1,1.732)--(1,-1.732);
\draw [black] (-1,1.732)--(1,0.4);
\draw [black] (-1,1.732)--(1,-0.4);
\filldraw [black] (-2,0) circle (2pt);
\node at (-2.3,0) {$z$};
\filldraw [black] (2,0) circle (2pt);
\node at (2.4,0) {$w$};
\draw [black] (-2,0)--(1,1.732);
\draw [black] (-2,0)--(1,-1.732);
\draw [black] (-2,0)--(-1,-1.732);
\draw [black] (-2,0)--(-1,1.732);
\draw [black] (-2,0)--(1,0.4);
\draw [black] (-2,0)--(1,-0.4);
\draw [black] (2,0)--(1,1.732);
\draw [black] (2,0)--(1,-1.732);
\draw [black] (2,0)--(-1,-1.732);
\end{tikzpicture}&
\begin{tikzpicture}[scale=0.6]
\filldraw [black] (0,0) circle (2pt);
\filldraw [black] (1,1.732) circle (2pt);
\filldraw [black] (1,-1.732) circle (2pt);
\draw [black] (1,1.732)--(0,0)--(1,-1.732);
\filldraw [black] (-2,0) circle (2pt);
\node at (-2.3,0) {$z$};
\filldraw [black] (2,0) circle (2pt);
\node at (2.4,0) {$w$};
\draw [black] (2,0)--(1,1.732);
\draw [black] (2,0)--(1,-1.732);
\draw [black] (-2,0)--(1,1.732);
\draw [black] (-2,0)--(0,0);
\draw [black] (-2,0)--(1,-1.732);
\node at (1,1) {$\vdots$};
\node at (1,-0.6) {$\vdots$};
\node at (1,0) {$s$};
\end{tikzpicture}\\
$U_9$ & $U_{10}$ & $U_{11}$ & $U_{12}=U_{12}(n,s)$
\end{tabular}
\caption{The graphs $U_1,\dots,U_{12}$.}
\label{F4}
\end{figure}
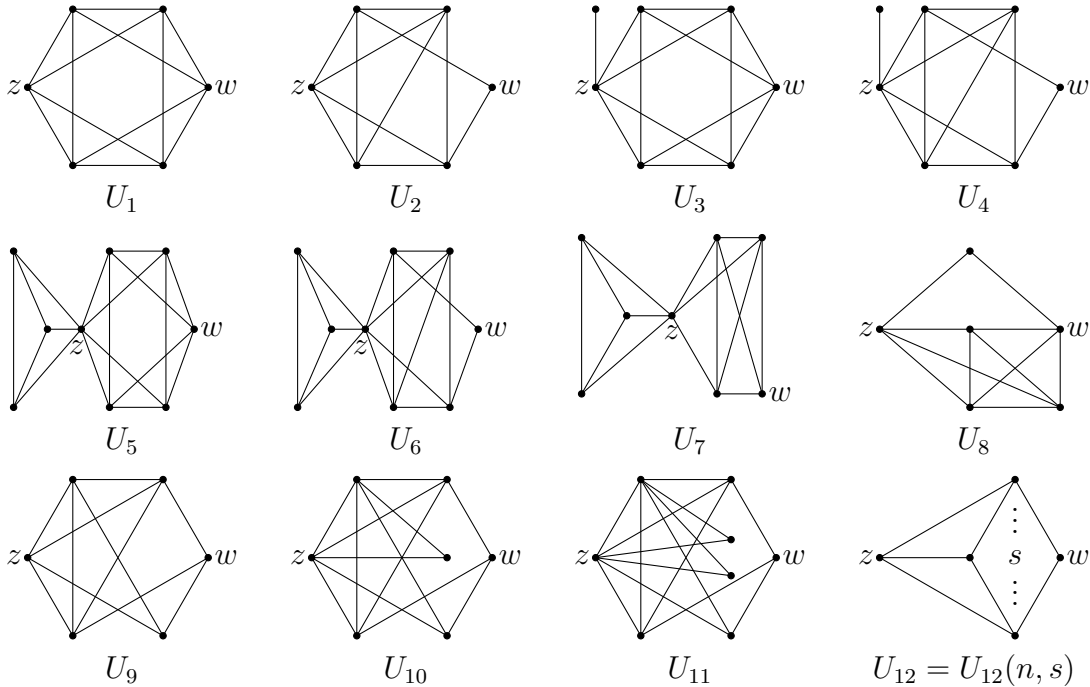

Let $U_1,\dots, U_{12}$ be the graphs displayed in Fig. 1. For $i=1,\dots, 11$, let $G_i$ be the $n$-vertex graph obtained from $U_i$ by adding all possible edges between $z$ and $\frac{n-|V(U_i)|}{4}$ vertex disjoint $K_4$, where $n\ge 7$ and $n-|V(U_i)|\equiv 0 \pmod 4$. Let $G_{12}=G_{12}(n,s)$ be the $n$-vertex graph obtained from $U_{12}(n,s)$ by adding all possible edges between $z$ and $\frac{n-|V(U_{12}(n,s))|}{4}$ vertex disjoint $K_4$, where $n\ge s+7$ and $s\ge 3$. Let $G_{13}$ be the graph obtained from $K_{3, n-3}$ by adding an edge in the partite set with $3$ vertices, where $n\ge 7$.

\begin{lemma}\label{930}
For $i=1,\dots, 13$, $q(G_i)<q(K_{1,1,n-2}^+)$.
\end{lemma}

\begin{lemma}\label{931}
If $n\ge 7$ and $G$ is one of the graphs $K_1\vee \frac{n-1}{4}K_4$ with $n\equiv 1 \pmod 4$,  $K_1\vee (K_1\cup \frac{n-2}{4}K_4)$ with $n\equiv 2 \pmod 4$, $K_1\vee \left(K_{1,1}\cup \frac{n-3}{4}K_4\right)$ with $n\equiv 3 \pmod 4$, or $K_1\vee (K_{1,s}^+\cup \frac{n-s-2}{4}K_4)$ with  $2\le s\le n-6$, then $q(G)<q(K_{1,1,n-2}^+)$.
\end{lemma}

We give the proof of Lemmas \ref{930} and \ref{931} in Appendices A and B, respectively.

\section{Proof of Theorem \ref{doubly21}}

Suppose that $G$ is a graph of order $n$ without isolated vertices that does not contain a trebly chorded cycle with chords incident to a vertex such that $q(G)$ is as large as possible.

Note that  $G$ is connected. Otherwise,   $q(G)=q(H_1)$ for some component $H_1$ of $G$.
Let $H_2$ be the another component of $G$. Let $u\in V(H_1)$ and $v\in V(H_2)$.
Note that $G+uv$ does not contain a trebly chorded cycle with chords incident to a vertex.
By Lemma \ref{addedges}, $q(G+uv)\ge q(H_1\cup H_2+uv)>q(H_1\cup H_2)=q(H_1)=q(G)$,
which is a contradiction.

Let $z$ be a vertex of $G$. Let $W=V(G)\setminus N_G[z]$, $Z=N_G(z)$, $Z_0=\{v:v\in Z,d_Z(v)=0\}$ and $Z_+=Z\backslash Z_0$.
Since $G$ does not contain a trebly chorded cycle with chords incident to a vertex,
$G[Z]$ is $P_5$-free, so every (connected) component of $G[Z_+]$ is $K_{1,s}$, $S_{n_1,n_2}$, $K_{1,t}^+$, $C_4$, $C_4^+$ or $K_4$ where $s\ge 1$, $n_1,n_2\ge 1$ and $t\ge 2$.
A star $K_{1,s}$ with $s\ge 3$  is said to be  big.

\begin{proof}[Proof of Theorem \ref{doubly21} for $n=6$]
Suppose that $n=6$.
As $K_1\vee(K_4\cup K_1)$ does not have such a trebly chorded cycle, we have $q(G)\ge q(K_1\vee(K_4\cup K_1))=8.2749$.

If $\Delta(G)\le 4$, then $q(G)\le 2\Delta(G)\le 8<8.2749$, a contradiction, so  $\Delta(G)=5$. Assume that $z$ is a vertex with $d(z)=\Delta(G)$.
As every component of $G[Z_+]$ is $K_{1,s}$ with $1\le s\le 4$, $P_4$, $S_{1,2}$, $K_{1,t}^+$ with $2\le t\le 4$, $C_4$, $C_4^+$ or $K_4$, we have  $e(N(z))\le 6$.
If $e(N(z))=6$, then $G\cong K_1\vee(K_4\cup K_1)$. It suffices to show that $e(N(z))<6$ is impossible.

If $e(N(z))=0,1$, then $G\cong K_{1,5}, K_{1,5}^+$; if $e(N(z))=2$, then $G\cong K_1\vee (2K_2\cup K_1), K_1\vee (K_{1,2}\cup 2K_1)$.  In either case, $G\subseteq K_1\vee S_{1,2}$.

If $e(N(z))=3$, then $G\cong K_1\vee(P_4\cup K_1), K_1\vee (K_{1,3}\cup K_1), K_1\vee (P_3\cup K_2), K_1\vee(K_3\cup 2K_1)$, so $G\subseteq K_1\vee S_{1,2}, K_{1,1,4}^+$.

If $e(N(z))=4$, then $G\cong  K_1\vee S_{1,2}, K_1\vee K_{1,4},  K_1\vee(K_{1,3}^+\cup K_1), K_1\vee(C_4\cup K_1), K_1\vee(K_3\cup K_2)$, so $G\subseteq K_1\vee S_{1,2}, K_{1,1,4}^+, K_1\vee(C_4^+\cup K_1),K_1\vee(K_3\cup K_2)$.

If $e(N(z))=5$, then $G\cong K_{1,1,4}^+$ or $K_1\vee (C_4^+\cup K_1)$.

Thus, if $e(N(z))<6$, then $G\subseteq K_1\vee S_{1,2}$, $K_{1,1,4}^+$, $K_1\vee(C_4^+\cup K_1)$ or  $K_1\vee(K_3\cup K_2)$.
By a direct calculation, $q(K_1\vee S_{1,2})=7.1156$, $q(K_{1,1,4}^+)=7.7588$,
$q(K_1\vee(C_4^+\cup K_1))=7.7264$ and $q(K_1\vee(K_3\cup K_2))=7$. By Lemma \ref{addedges}, $q(G)<8.2749$, a contradiction.
\end{proof}

\begin{proof}[Proof of Theorem \ref{doubly21} for $n\ge 7$]
Suppose that $n\ge 7$.
As $K_{1,1,n-2}^+$ does not have such a trebly chorded cycle, we have $q(G)\ge q(K_{1,1,n-2}^+)>n+2-\frac{4}{n+2}$ by Lemma \ref{712}.

\begin{claim}\label{c31}
For $w\in W$ and a component $H$ of $G[Z_+]$, we have
\begin{itemize}
\item $d_H(w)\le 2$ if $H$ is a big star and $w$ is adjacent to the center of $H$,
\item $d_H(w)\le 3$ if $H$ is $S_{n_1,n_2}$ with $n_1+n_2\ge 3$ where $n_1,n_2\ge1$, moreover, $w$ is adjacent to at most one leaf of $H$,
\item $d_H(w)\le 3$ if $H$ is $K_{1,s}^+$ with $s\ge 3$, moreover, $w$ is adjacent to at most one pendant vertex of $H$ and $d_H(w)=1$ if $w$ is adjacent to the center of $H$,
\item $d_H(w)\le 2$ with equality if and only if each neighbor of $w$ has degree two in $G[H]$ if $H$ is $C_4^+$,
\item $d_H(w)\le 1$ if $H$ is $K_4$.
 \end{itemize}
\end{claim}

\begin{proof}
Suppose that $H$ is a big star.
Let $u$ be the center of $H$, and let $u_1,u_2$ and $u_3$ be leaves of $H$.
If $u_1,u_2\in N_{H}(w)$, then $zu_1wu_2uu_3z$ is a trebly chorded cycle with  chords $zu, u_1u$ and $wu$ (incident to $u$), a contradiction.

Suppose that $H$ is $S_{n_1,n_2}$ with $n_1+n_2\ge 3$. Let $V(H)=\{u,u_1,\dots,u_{n_1},v,v_1,\dots,v_{n_2}\}$, where $N_H(u)=\{u_1,\dots,u_{n_1},v\}$ and $N_H(v)=\{v_1,\dots,v_{n_2},u\}$.
If $w$ is adjacent to two leaves which are the neighbors of some center of $H$, say $u_1,u_2\in N_{H}(w)$, then $zu_1wu_2uvv_1z$ is a trebly chorded cycle with  chords $u_2z, uz$ and $vz$ (incident to $z$), a contradiction.
Suppose that $w$ is adjacent to two leaves which are neighbors of $u$ and $v$ respectively, say  $u_1,v_1\in N_{H}(w)$ and $n_1\ge n_2$, then $zu_2uu_1wv_1vz$ is a trebly chorded cycle with  chords $u_1z, uz$ and $v_1z$ (incident to $z$), a contradiction.
Thus $w$ is adjacent to at most one leaf of $H$.
So $d_H(w)\le 3$.

Suppose that $H$ is $K_{1,s}^+$ with $s\ge 3$. Let $V(H)=\{u,u_1,\dots,u_{s}\}$, where $u$ is the center of $H$ and $u_1$ adjacent to $u_2$.
Suppose first that $w$ is adjacent to $u$.
We claim that $d_H(w)=1$. Otherwise,  $d_H(w)\ge 2$.
If $w$ is adjacent to some pendant vertex of $H$, say $u_3$, then $zu_3wuu_1u_2z$ is a trebly chorded cycle with  chords $zu, u_2u$ and $u_3u$ (incident to $u$), a contradiction.
If $w$ is not adjacent to any pendant vertex of $H$, then we can assume that $u_1w\in E(G)$, so $zu_3uwu_1u_2z$ is a trebly chorded cycle with  chords $zu, u_1u$ and $u_2u$ (incident to $u$), a contradiction. Thus $d_H(w)=1$, as claimed.
Suppose next that $w$ is not adjacent to $u$.
If there are two pendant vertices adjacent to $w$, say $u_3$ and $u_4$, then $zu_4wu_3uu_1u_2z$ is a trebly chorded cycle with  chords $u_1z, uz$ and $u_3z$ (incident to $z$), a contradiction. Thus $w$ is adjacent to at most one pendant vertex of $H$.
So $d_{H}(w)\le 3$ with equality if only if $w$ is adjacent to $u_1,u_2$ and one leaf of $H$.

Suppose that $H$ is $C_4^+$ with $C_4=u_1u_2u_3u_4u_1$ and $u_1$ adjacent to $u_3$.
Suppose that $d_{H}(w)\ge 3$. Then there exists at least one neighbor of $w$ which has degree three in $H$ and at least one neighbor of $w$ which has degree two in $H$. Assume that $u_1,u_2\in N_{H}(w)$.
Then $zu_1wu_2u_3u_4z$ is a trebly chorded cycle with  chords $u_2u_1, u_3u_1$ and $u_4u_1$ (incident to $u_1$), a contradiction.
So $d_{H}(w)\le 2$.
Moreover, if $u_1,u_3\in N_{H}(w)$, then $zu_2u_1wu_3u_4z$ is a trebly chorded cycle with  chords $zu_1, u_3u_1$ and $u_4u_1$ (incident to $u_1$), a contradiction. Thus $d_{H}(w)= 2$ if and only if $u_2,u_4\in N_{H}(w)$.

Suppose that $H$ is $K_4$ with $V(H)=\{u_1,u_2,u_3,u_4\}$. If $u_1,u_2\in N_{H}(w)$, then $zu_1wu_2u_3u_4z$ is a trebly chorded cycle with  chords $u_2u_1, u_3u_1$ and $u_4u_1$ (incident to $u_1$), a contradiction.
\end{proof}

\begin{claim}\label{c2} Let $w\in W$ and let $H$ be a component of $G[Z_+]$ which is not a star. Suppose that $d_{H}(w)\ge 1$. Then $d_{H'}(w)=0$ for any other component $H'$ of $G[Z_+]$. Moreover, $d_{Z_0}(w)=0$ if one of the following conditions is satisfied:
\begin{itemize}
\item $H$ contains $C_4$;
\item $H\cong K_{1,s}^+$ where $s\ge 3$ and  $w$ is not adjacent to the center of $H$;
\item $H\cong S_{n_1,n_2}$ where $n_1,n_2\ge 1$ and $w$ is adjacent to some leaf of $S_{n_1,n_2}$.
\end{itemize}
\end{claim}

\begin{proof} 
Suppose that $d_{H'}(w)\ge 1$.
Assume that $a\in N_{H'}(w)$ and $ab\in E(H')$.

Suppose that $H\cong K_3$. Let $V(H)=\{u_1,u_2,u_3\}$. Assume that $w$ is adjacent to $u_1$. Then $zu_3u_2u_1wabz$ is a trebly chorded cycle with  chords $u_1z, u_2z$ and $az$ (incident to $z$), a contradiction.

Suppose that $H\ncong K_3$. Since $H$ is not a star, $P_4\subseteq H$. Let $P_4=u_1u_2u_3u_4$.
If $w$ is adjacent to an end vertex of $P_4$, say $u_1$, then $zu_4u_3u_2u_1wabz$ is a trebly chorded cycle with  chords $u_1z, u_2z$ and $u_3z$ (incident to $z$), a contradiction.
If $w$ is adjacent to an internal vertex of $P_4$, say $u_2$, then $zu_4u_3u_2wabz$ is a trebly chorded cycle with  chords $u_2z, u_3z$ and $az$ (incident to $z$), a contradiction.

Moreover,
suppose that $H$ satisfies one of the above conditions and  $d_{Z_0}(w)\ge 1$.
Assume that $a^*\in N_{Z_0}(w)$. Let $u_1$ be a neighbor of $w$ in $H$. By the condition on $H$, 
we can find a copy of $P_4:=u_1u_2u_3u_4$ in $H$. 
Then $zu_4u_3u_2u_1wa^*z$ is a trebly chorded cycle with  chords $u_1z, u_2z$ and $u_3z$ (incident to $z$), a contradiction.
\end{proof}

\begin{claim}\label{new}
If $\Delta(G)\le n-2$, then $\eta_G(z)\le n+2-\frac{4}{n+2}$ for any vertex $z$ with $d(z)\le n-4$ of $G$.
\end{claim}
\begin{proof}
If $d(z)\le 3$, then
\[
\eta_{G}(z)\le d(z)+\Delta(G)\le 3+n-2=n+1<n+2-\frac{4}{n+2}.
\]

Suppose that $d(z)=4$. Then  $e(Z)\le 6$.
If $e(Z)\le 3$, from \eqref{q12}, we have
\[
\eta_{G}(z)\le n+\frac{3}{2}< n+2-\frac{4}{n+2}.
\]
Suppose that $e(Z)=4$. Then $G[Z]\cong K_{1,3}^+, C_4$.
In the former case,  from Claims \ref{c31} and \ref{c2}, we have $d_Z(w)\le 3$ for each $w\in W$,
so  $e(Z,W)\le 3|W|=3(n-5)=3n-15$. In the latter case,
we claim that $d_Z(w_1)+d_Z(w_2)\le 5$ for any $w_1,w_2\in W$.
Otherwise, $d_Z(w_1)+d_{Z}(w_2)\ge 6$, and there exists at least one vertex in $W$ adjacent to three vertices, $u_1,u_2,u_3$ in $Z$.
Note that $N_{Z}(w_1)\cap N_{Z}(w_2)\ge 2$.
If $u_1,u_2\in N_{Z}(w_1)\cap N_{Z}(w_2)$, then $zu_4u_3w_1u_1w_2u_2z$ is a trebly chorded cycle with  chords $zu_1, u_2u_1$ and $u_4u_1$ (incident to $u_1$), a contradiction;
if $u_1,u_3\in N_{Z}(w_1)\cap N_{Z}(w_2)$, then $zu_4u_3w_2u_1w_1u_2z$ is a trebly chorded cycle with  chords $zu_1, u_2u_1$ and $u_4u_1$ (incident to $u_1$), a contradiction. So  $d_Z(w_1)+d_{Z}(w_2)\le 5$, as claimed.  Then $e(Z,W)\le \frac{5}{2}|W|=\frac{5}{2}(n-5)=\frac{5}{2}n-\frac{25}{2}$ in the latter case.
In either case, $e(Z,W)\le 3n-15$, so,
from \eqref{q11}, we have
\[
\eta_{G}(z)\le 4+\frac{4+2\times 4+3n-15}{4}=\frac{3}{4}n+\frac{13}{4}<n+2-\frac{4}{n+2}.
\]
Suppose that $e(Z)=5,6$. Correspondingly, from Claims \ref{c31} and \ref{c2}, we have $G[Z]\cong C_4^+$, $d_Z(w)\le 2$ for each $w\in W$, $e(Z,W)\le 2|W|=2(n-5)=2n-10$, from \eqref{q11}, we have
\[
\eta_{G}(z)\le 4+\frac{4+2\times 5+2n-10}{4}=\frac{1}{2}n+5<n+2-\frac{4}{n+2},
\]
or $G[Z]\cong K_4$, $d_Z(w)\le 1$ for each $w\in W$, $e(Z,W)\le  |W|=n-5$,
from \eqref{q11}, we have
\[
\eta_{G}(z)\le 4+\frac{4+2\times 6+n-5}{4}=\frac{1}{4}n+\frac{27}{4}<n+2-\frac{4}{n+2}.
\]

Suppose that $d(z)=5$.  It is easy to see that $e(Z)\le 6$.
If $e(Z)\le 4$, from \eqref{q12}, we have
\[
\eta_{G}(z)\le n+\frac{8}{5}\le n+2-\frac{4}{n+2}.
\]
If $e(Z)=5$, then $G[Z]\cong K_{1,4}^+, K_1\cup C_4^+$, so,
from Claims \ref{c31} and \ref{c2}, we have $d_Z(w)\le 3$ for each $w\in W$, implying that
 $e(Z,W)\le 3|W|=3(n-6)=3n-18$. From \eqref{q11}, we have
\[
\eta_{G}(z)\le 5+\frac{5+2\times 5+3n-18}{5}=\frac{3n}{5}+\frac{22}{5}< n+2-\frac{4}{n+2}.
\]
If $e(Z)=6$, then $G[Z]\cong K_1\cup K_4$, so from Claims \ref{c31} and \ref{c2}, we have $d_Z(w)\le 1$ for each $w\in W$. Then $e(Z,W)\le  |W|=n-6$, and from \eqref{q11}, we have
\[
\eta_{G}(z)\le 5+\frac{5+2\times 6+n-6}{5}=\frac{n}{5}+\frac{36}{5}<n+2-\frac{4}{n+2}.
\]

Suppose that $6\le d(z)\le n-4$. Then $n\ge 10$.
By Lemma \ref{ee}, $e(G)\le 4n-16$.
As $G$ does not contain a trebly cycle with chords incident to a vertex, $G[Z]$ is $P_5$-free, so
from Lemma \ref{p5}, we have $e(Z)\le \frac{3}{2}d(z)$.
Thus
\begin{align*}
\sum_{u\in Z}d(u)&=d(z)+2e(Z)+e(Z,W)\\
&\le d(z)+2e(Z)+e(G)-d(z)-e(Z)\\
&=e(G)+e(Z)\\
&\le 4n+\frac{3}{2}d(z)-16.
\end{align*}
It follows that
\[
\eta_G(z)\le d(z)+\frac{4n+\frac{3}{2}d(z)-16}{d(z)}=d(z)+\frac{3}{2}+\frac{4n-16}{d(z)}.
\]
Let $f(x)=x+\frac{3}{2}+\frac{4n-16}{x}$ for $x\in[6,n-4]$.
As
$f''(x)=\frac{8n-32}{x^3}> 0$,
\[
\eta_G(z)\le \max\{f(6),f(n-4)\}=n+\frac{3}{2}<n+2-\frac{4}{n+2}. \qedhere
\]
\end{proof}

If $\Delta(G)\le n-4$, then we have from Claim \ref{new} and Lemma \ref{q1} and {712}, we have $q(G)<q(K_{1,1,n-2}^+$,   a contradiction.

From now on we assume that $\Delta(G)\ge n-3$, and $z$ is an arbitrary vertex of $G$ such that $d(z)=\Delta(G)$. Denote by $Z^*$ the set of vertices of the components in $G[Z]$ which does not contain the neighbors of the vertices of $W$.
Let $\mathbf{x}$ be the Perron vector of $Q(G)$.

\begin{claim}\label{nok4}
There is at most one component that does not contain $C_4$ in $G[Z^*]$, and any other component
is $K_4$.
Moreover, if there is one component that does not contain $C_4$ in $G[Z^*]$, then the component is $K_1$, $K_{1,1}$ or $K_{1,s}^+$ for some $s\ge 2$.
\end{claim}
\begin{proof}
We first claim that there is no double star component in $G[Z^*]$. For, suppose that  $S_{b_1,b_2}$
is a  component in  $G[Z^*]$, in which  $v_1$ and $v_2$ are the centers, and $v_{i,1},\dots,v_{i,b_i}$ are the leaves adjacent to $v_i$ for any $i=1,2$.
Assume that $x_{v_1}\ge x_{v_2}$.
Let $G'=G-\{v_2v_{2,j}:1\le j\le b_{2}\}+\{v_1v_{2,j}:1\le j\le b_{2}\}$.
Obviously, $G'$ does not contain a trebly chorded cycle with chords incident to a vertex.
From Lemmas \ref{addedges} and \ref{perron}, we have $q(G)<q(G')$, a contradiction.

Note that if $uv\in E(G)$ and $N[u]\subseteq N[v]$, then since
\begin{align*}
q(G)x_u=d(u)x_u+\sum_{w\in N(u)}x_w\\
q(G)x_v=d(v)x_v+\sum_{w\in N(v)}x_w,
\end{align*}
we have
\[
\left(q(G)-d(u)+1\right)\left(x_v-x_u\right)=\left(d(v)-d(u)\right)x_u+\sum_{w\in N[v]\backslash N[u]}x_w,
\]
so $x_u\le x_v$. If $K_{1,s}$ is a component of $G[Z^*]$, $u$ is the center and $u_{1},\dots,u_{s}$ are the leaves of $K_{1,s}$, then $x_{u}\ge x_{u_{1}}$. Thus, if $K_{1,s}$ is a component of $G[Z^*]$, then the entry of the Perron vector is maximum  at the center among vertices of $K_{1,s}$. Similarly,
 if  $K_{1,s}^+$ is a component of $G[Z^*]$ then the entry of the Perron vector is maximum at the center among the vertices of $K_{1,s}^+$.


We next claim that there exists at most one star component in $G[Z^*]$.
For, suppose that there are two star components in $G[Z^*]$.
Assume that $x_{u_1}$  is maximum among the centers of all star components of $G[Z^*]$.
Let $G''$ be the graph obtained from $G$ by deleting all edges of the star components of $G[Z^*]$ except the one with center $u_1$ and adding edges between $u_1$ and all arising isolated vertices in $Z^*$.
Obviously, $G''$ does not contain a trebly chorded cycle with chords incident to a vertex.
From Lemmas \ref{addedges} and \ref{perron}, we have $q(G)<q(G'')$, a contradiction.

Now we claim that any star component of $G[Z^*]$ must be $K_1, K_{1,1}$. For,
suppose that $K_{1,s}$ is the star component of $G[Z^*]$ with some $s\ge 2$. Let $v_1^*,v_2^*$ be two leaves of $K_{1,s}$.
Obviously, $G+v_1^*v_2^*$ does not contain a trebly chorded cycle with chords incident to a vertex.
From Lemma \ref{addedges}, we have $q(G)<q(G+v_1^*v_2^*)$, a contradiction.


Finally, we claim that there exists at most one $K_{1,s}^+$ component in $G[Z^*]$ for some $s\ge 2$.
For, suppose that there are two components $K_{1,s_1}^+$ and $K_{1,s_2}^+$ with centers  $u_1$ and $u_2$, respectively.  Assume that $x_{u_1}\ge x_{u_2}$.
Let $G^*$ be the graph obtained from $G$ by deleting all edges of $K_{1,s_2}^+$ and adding edges between $u_1$ and $V(K_{1,s_2}^+)$.
Obviously, $G^*$ does not contain a trebly chorded cycle with chords incident to a vertex.
From Lemmas \ref{addedges} and \ref{perron}, we have $q(G)<q(G^*)$, a contradiction.

Suppose that there are two components, say $H_1$ and $H_2$,  which do not contain $C_4$ in $G[Z^*]$,
According to the above claims, we can assume that $H_1\cong K_{1,t}^+$ and $H_2\cong K_{1,s}$, where $t\ge 2$ and $s=0,1$. Let $z_i$ be the center of $H_i$ with $i=1,2$.  For $\{i,j\}=\{1,2\}$ with $x_{z_i}\ge x_{z_j}$.
Let $G^{**}$ be the graph obtained from $G$ by deleting all edges of $H_j$ and adding edges between $z_i$ and $V(H_j)$.
Obviously, $G^{**}$ does not contain a trebly chorded cycle with chords incident to a vertex.
From Lemmas \ref{addedges} and \ref{perron}, we have $q(G)<q(G^{**})$, a contradiction.
Thus,  there is at most one component that does not contain $C_4$ in $G[Z^*]$. Moreover,
by the above argument, if there is  one component that does not contain $C_4$ in $G[Z^*]$, then the component is $K_1$, $K_{1,1}$ or $K_{1,s}^+$ where $s\ge 2$.

From Lemma \ref{addedges} and our choice of $G$,  any  component of $G[Z^*]$ which contains $C_4$ must be  $K_4$
%
%
\end{proof}

Suppose that $d(z)=n-1$.
If $G[Z]$ does not contain $C_4$, from Claim \ref{nok4} and $n\ge 7$, we have $G\cong K_{1,1,n-2}^+$.
Suppose that $G[Z]$ contains $C_4$. From Claim \ref{nok4}, each component $H$ in $G[Z]$ which contains $C_4$ is $K_4$, so $G$ is isomorphic to $K_1\vee \frac{n-1}{4}K_4, K_1\vee \left(K_1\cup \frac{n-2}{4}K_4\right), K_1\vee \left(K_{1,1}\cup \frac{n-3}{4}K_4\right)$ or $K_1\vee \left(K_{1,s}^+\cup \frac{n-s-2}{4}K_4\right)$ where $s\le n-6$.
From Lemma \ref{931}, we have $q(G)<q(K_{1,1,n-2}^+)$, a contradiction.

Suppose that $d(z)=n-3, n-2$.

\begin{case}\label{a1}
$d(z)=n-3$.
\end{case}
It is  known that $q(G)\le 2\Delta(G)$.
If $n=7$, we have $q(G)\le 2(n-3)<n+2-\frac{4}{n+2}$, a contradiction.
Suppose that $n\ge 8$.
In this case, $|W|=2$.
Let $W=\{w_1,w_2\}$.

\begin{claim}\label{2221}
$e(Z_+,W)\ge 1$.
\end{claim}

\begin{proof} Suppose on the contrary that $e(Z_+,W)=0$. Then $1\le |Z_0|\le n-3$.
Note that $d(u)\le 3$ for any $u\in Z_0$.
Suppose first that $d(w_i)= n-3$ for some $i=1,2$.
Since $\Delta(G)=n-3$, we have
\[
\eta_{G}(w_i)\le n-3+\frac{3(n-4)+n-3}{n-3}= n+1-\frac{3}{n-3}<n+2-\frac{4}{n+2}.
\]
Suppose next that $d(w_i)\le n-4$ for each $i=1,2$. From Claim \ref{new}, we have $\eta_G(w_i)\le n+2-\frac{4}{n+2}$.

As $G$ does not contain a trebly cycle with chords incident to a vertex, $G[Z_+]$ is $P_5$-free.
From Lemma \ref{p5}, we have $e(Z)\le \frac{3}{2}\left(n-3-|Z_0|\right)$.
Since $e(W,Z)\le 2|Z_0|$, we have
\[
\eta_{G}(z)\le n-3+\frac{n-3+3\left(n-3-|Z_0|\right)+2|Z_0|}{n-3}=n+1-\frac{|Z_0|}{n-3}\le n+1-\frac{1}{n-3}<n+2-\frac{4}{n+2}.
\]

For any $v\in Z$, we have $d(v)\le n-4$. Combining Claim \ref{new}, Lemma \ref{712}, and the above proof, we have
$\max_{v\in V(G)}\eta_{G}(v)\le n+2-\frac{4}{n+2}<q(K_{1,1,n-2}^+)$.
From Lemma \ref{q1}, we have $q(G)<q(K_{1,1,n-2}^+)$, a contradiction.
\end{proof}

Let $Z_1$ be the set of the vertices of the star components in $G[Z_+]$, and let $Z_2=Z_+\backslash Z_1$.

\begin{claim}\label{51}
$\eta_{G}(z)\le n+2-\frac{4}{n+2}$.
\end{claim}
\begin{proof}
As $G$ does not contain a trebly cycle with chords incident to a vertex, $G[Z_2]$ is $P_5$-free, so
from Lemma \ref{p5}, we have $e(Z_2)\le \frac{3}{2}|Z_2|$.
We divide the proof into following three cases:

\noindent
{\bf Case i.} Both $w_1$ and $w_2$ have at least one neighbor in $Z_1$.

From Claim \ref{c2}, the neighbors of $w_i$ in $Z$ are all in $Z_1$ for $i=1,2$.
From \eqref{q11}, we have
\begin{align*}
\eta_G(z)&\le  d(z)+\frac{d(z)+2e(Z_2)+2e(Z_1)+e(Z_0,W)+e(Z_1,W)}{d(z)}\\
&\le d(z)+\frac{d(z)+3(d(z)-|Z_0|-|Z_1|)+2(|Z_1|-1)+2|Z_0|+2|Z_1|}{d(z)}\\
&=d(z)+4+\frac{|Z_1|-|Z_0|-2}{d(z)}.
\end{align*}
If $|Z_1|\le n-5$, then $\eta_G(z)\le n+2-\frac{|Z_0|+4}{n-3}\le n+2-\frac{4}{n-3}<n+2-\frac{4}{n+2}$.
If $|Z_1|=n-4$, then $|Z_0|=1$, so $\eta_G(z)\le n+2-\frac{4}{n-3}<n+2-\frac{4}{n+2}$.
Suppose that $|Z_1|=n-3$. Then $Z_0=\emptyset$.
Let $t$ be the number of the star components in $G[Z_1]$.
If $t=1$, then $G[Z_1]$ is a big star $K_{1,n-4}$ as $n\ge 8$, so, from Claim \ref{c31}, we have  $e(Z_1,W)\le 2(n-4)$, implying that  $2e(Z_1)+e(Z_1,W)\le 4(n-4)$.
If $t\ge 2$, then $2e(Z_1)+e(Z_1,W)\le 2(|Z_1|-t)+2|Z_1|=4n-12-2t\le 4(n-4)$.
Thus $2e(Z_1)+e(Z_1,W)\le 4(n-4)$, and then $\eta_G(z)\le n-3+1+\frac{4(n-4)}{n-3}=n+2-\frac{4}{n-3}<n+2-\frac{4}{n+2}$.

\noindent
{\bf Case ii.}  Exactly  one of $w_1$ and $w_2$, say $w_1$, has neighbors in $Z_1$.

From Claims \ref{c31} and \ref{c2}, we have $e(\{w_1\},Z)\le |Z_0|+|Z_1|$ and $e(\{w_2\},Z)\le |Z_0|+4$.
From \eqref{q11}, we have
\begin{align*}
\eta_G(z)&\le  d(z)+\frac{d(z)+2e(Z_2)+2e(Z_1)+e(\{w_1\},Z)+e(\{w_2\},Z)}{d(z)}\\
&\le d(z)+\frac{d(z)+3(d(z)-|Z_0|-|Z_1|)+2(|Z_1|-1)+2|Z_0|+|Z_1|+4}{d(z)}\\
&=d(z)+4-\frac{|Z_0|-2}{d(z)}\le n+1+\frac{2}{n-3}\\
&<n+2-\frac{4}{n+2}.
\end{align*}

\noindent
{\bf Case iii.} Neither $w_1$ nor $w_2$  has neighbors in  $Z_1$.

From \eqref{q11}, we have
\begin{align*}
\eta_G(z)&\le  d(z)+\frac{d(z)+2e(Z_+)+e(Z_0,W)+e(Z_2,W)}{d(z)}\\
&\le d(z)+\frac{d(z)+3(d(z)-|Z_0|)+2|Z_0|+e(Z_2,W)}{d(z)}\\
&=d(z)+4+\frac{e(Z_2,W)-|Z_0|}{d(z)}.
\end{align*}
If $e(Z_2,W)\le n-7$, then $\eta_G(z)\le n+2-\frac{4}{n-3}<n+2-\frac{4}{n+2}$.
Suppose that $e(Z_2,W)\ge n-6$.
From Claims \ref{c31} and \ref{c2}, we have $e(Z_2,W)\le 8$.
So $8\le n\le 14$.

Assume that $e(Z_2,\{w_1\})\ge e(Z_2,\{w_2\})$. By Claim \ref{c31}, $e(Z_2,\{w_1\})\le 4$, and
if $e(Z_2,\{w_1\})=4$, then  there exists at least one component  $C_4$ or $P_4$ in $G[Z_+]$, so $2e(Z_+)\le 3d(z)-3|Z_0|-4$.
Suppose that $e(Z_2,\{w_1\})=4$.
From \eqref{q11}, we have
\begin{align*}
\eta_G(z)&\le  d(z)+\frac{d(z)+3d(z)-3|Z_0|-4+8+2|Z_0|}{d(z)}\\
&=d(z)+4+\frac{4-|Z_0|}{d(z)}\le n+1+\frac{4}{n-3}\\
&<n+2-\frac{4}{n+2}
\end{align*}
for $n\ge 10$.
By a direct calculation, we have  $\eta_G(z)\le 5+\frac{5+2\times 4+8}{5}=\frac{46}{5}<10-\frac{2}{5}$ if $n=8$, and
$\eta_G(z)\le 6+\frac{6+2\times 5+8}{6}=10<11-\frac{4}{11}$ if $n=9$.
So $\eta_G(z)<n+2-\frac{4}{n+2}$ for $8\le n\le 14$.

Suppose that $e(Z_2,\{w_1\})\le 3$. Then $e(Z_2,W)\le 6$ and $n\le 12$. We claim that $2e(H)+e(V(H),W)\le 4|V(H)|-2$ for any component $H$ in $G[Z]$.
If $H$ is $K_1$, then $2e(H)+e(V(H),W)=e(V(H),W)\le 2= 4|V(H)|-2$;
if $H$ is a star $K_{1,s}$ with $s\ge 1$, then $2e(H)+e(V(H),W)=2e(H)= 2|V(H)|-2<4|V(H)|-2$;
if $H$ is a double star, then $2e(H)+e(V(H),W)\le 2(|V(H)|-1)+6=2|V(H)|+4<4|V(H)|-2$;
if $H\cong K_4$, then $2e(H)+e(V(H),W)\le 2\times 6+2=4|V(H)|-2$;
if $H\cong C_4^+$, then $2e(H)+e(V(H),W)\le 2\times 5+4=4|V(H)|-2$;
if $H\cong C_4$, then $2e(H)+e(V(H),W)\le 2\times 4+6=4|V(H)|-2$.
Suppose that $H\cong K_{1,s}^+$.  If $s\ge 3$, then $2e(H)+e(V(H),W)\le 2|V(H)|+6\le 4|V(H)|-2$.
If $s=2$, then $H\cong K_3$, so $e(V(H),W)\le 4$. Otherwise,  $e(V(H),\{w_1\})=3$ and $e(V(H),\{w_2\})\ge 2$.
Assume that $V(H)=\{v_1,v_2,v_3\}$ and $N_{H}(w_2)\supseteq \{v_2,v_3\}$. Then $zv_1w_1v_2w_2v_3z$ is a trebly chorded cycle with  chords $zv_2, v_1v_2$ and $v_3v_2$ (incident to $v_2$), a contradiction.
So $2e(H)+e(V(H),W)\le 2\times 3+4=4|V(H)|-2$.
Thus  $2e(H)+e(V(H),W)\le 4|V(H)|-2$ for any component $H$ in $G[Z]$.
Let $\mathcal{H}$ be the set of the components of $G[Z]$.
Since
\begin{align*}
\sum_{v\in Z}d(v)&=d(z)+2e(Z)+e(Z,W)\\
&=d(z)+\sum_{H\in \mathcal{H}}(2e(H)+e(V(H),W))\\
&\le d(z)+\sum_{H\in \mathcal{H}}(4|V(H)|-2)\\
&=5d(z)-2|\mathcal{H}|,
\end{align*}
we have
\[
\eta_G(z)\le d(z)+5-\frac{2|\mathcal{H}|}{d(z)}\le n+2-\frac{4}{n-3}<n+2-\frac{4}{n+2}
\]
for $|\mathcal{H}|\ge 2$.
We are left with $|\mathcal{H}|=1$, say $\mathcal{H}=\{H\}$.
Since $n\ge 8$, we have $H$ does not contain $C_4$ and $|V(H)|\ge 5$.
Note that  $H$ is not a star.  Then $H$ is a double star or $K_{1,s}^+$ with $s\ge 4$.
If $H$ is a double star, then $2e(H)+e(V(H),W)\le 2(|V(H)|-1)+6=2|V(H)|+4$;
if $H\cong K_{1,s}^+$, then $2e(H)+e(V(H),W)\le 2|V(H)|+6$.
So $2e(H)+e(V(H),W)\le 2|V(H)|+6$.
Then
\[
\eta_G(z)\le d(z)+1+\frac{2d(z)+6}{d(z)}=n+\frac{6}{n-3}<n+2-\frac{4}{n+2}. \qedhere
\]
\end{proof}

By Claims \ref{new} and \ref{51}, we have $\eta_{G}(v)\le n+2-\frac{4}{n+2}$ for any $v\in V(G)$.
By Lemmas \ref{q1} and \ref{712}, we have $q(G)<q(K_{1,1,n-2}^+)$, a contradiction.

\begin{case}\label{a1}
$d(z)=n-2$.
\end{case}

Let $W=\{w\}$.

\begin{claim}\label{6662}
Exactly one component of $G[Z_+]$ contains the neighbors of $w$.
\end{claim}
\begin{proof}
We first show that $d_{Z_+}(w)\ge 1$. Otherwise, as $G$ is connected, we have $d_{Z_0}(w)\ge 1$.
Let $G'=G+wz$.
It is easy to see that $G'$ does not contain a trebly chorded cycle with chords incident to a vertex.
From Lemma \ref{addedges}, we have $q(G')>q(G)$, a contradiction.

Suppose that the neighbors of $w$ are contained in at least two components of $G[Z_+]$.
Let $Z_1$ be the set of vertices in some  star component of $G[Z_+]$.
From Claim \ref{c2}, $N(w)\subseteq Z_1\cup Z_0$.
Note that $e(Z_1)\le |Z_1|-2$ and $e(Z,W)\le |Z_0|+|Z_1|$.
From Lemma \ref{p5}, we have $e(Z)\le |Z_1|-2+\frac{3}{2}(n-2-|Z_0|-|Z_1|)$.
Then
\begin{align*}
\eta_G(z)&\le n-2+\frac{n-2+2|Z_1|-4+3n-6-3|Z_0|-3|Z_1|+|Z_0|+|Z_1|}{n-2}\\
&=n+2-\frac{2|Z_0|+4}{n-2}<n+2-\frac{4}{n+2}.
\end{align*}

Suppose that  $d(w)=n-3,n-2$.
From Claim \ref{c31}, $d_{Z_1}(v)\le 2$ for any $v\in N(w)$. Suppose that $d_{Z_1}(v)=2$ for some $v\in N(w)$. Let $N_{Z_1}(v)=\{v_1,v_2\}$.   Let $\{v_3,v_4\}\subseteq N(w)$ such that $v_3$ and $v_4$ lie in a component of $G[Z_1]$.
Then $zv_2vv_1wv_3\dots v_4z$ is a trebly chorded cycle with chords $vz, v_1z$ and $v_3z$ (incident to $z$), a contradiction. So $d_{Z_1}(v)\le 1$ for any $v\in N(w)$, that is,
$d(v)\le 3$ for any $v\in N(w)$.
Then
\[
\eta_G(w)\le d(w)+3\le n+1<n+2-\frac{4}{n+2}.
\]

Since there are at least two star components of $G[Z_+]$,  we have $d(v)\le n-5$ if $v\in Z\backslash Z_1$, $d(v)\le n-3$ if $v\in Z_1$ with equality if and only if $v$ is the center of $K_{1,n-5}$ and $wv\in E(G)$.
If there is some vertex $v\in Z_1$ with degree $n-3$, then $G[Z]\cong K_2\cup K_{1,n-5}$, $v$ is the center of $K_{1,n-5}$ and $v$ adjacent to $w$,
so
\[
\eta_G(v)\le n-3+\frac{n-2+3(n-5)+n-2}{n-3}=n+2-\frac{4}{n-3}<n+2-\frac{4}{n+2}.
\]

As above, we have a contradiction from Claim \ref{new} and Lemmas \ref{q1} and \ref{712}.
\end{proof}

By Claim \ref{6662}, all neighbors of $w$ lie in in one component, any $H$
of  $G[Z_+]$.

\begin{claim}\label{211}
$d_{H}(w)\ge 2$.
\end{claim}
\begin{proof}
Suppose that this is not true. Then $d_{H}(w)=1$. Let $v\in N_{Z_+}(w)$ where $v\in V(H)$.
If  $x_z\ge x_v$, then $G':=G-wv+wz$ does not contain a trebly chorded cycle with chords incident to a vertex, but from Lemma \ref{addedges}, we have $q(G')>q(G)$, a contradiction. If suffices to show $x_z\ge x_v$.

Suppose first that $d(v)=\Delta(G)=n-2$. Then $|N(z)\cap N(v)|=n-4$.
Let $N(z)\cap N(v)=\{v_1,v_2,\ldots,v_{n-4}\}$ and let $u$ be the neighbor of $z$ which is not adjacent to $v$.
Since
\begin{align*}
\left(q(G)-n+2\right)x_z&=x_v+\sum_{i=1}^{n-4}x_{v_i}+x_u,\\
\left(q(G)-n+2\right)x_v&=x_z+\sum_{i=1}^{n-4}x_{v_i}+x_w,
\end{align*}
we have
\[
\left(q(G)-n+3\right)\left(x_z-x_v\right)=x_u-x_w.
\]
According to the possible structure of $H$, it can be known that $u\in Z_0$ or $V(H)$.
So $d_Z(u)=0$ if $u\in Z_0$ and $d_Z(u)=1$ if $u\in V(H)$.
Note that $d(w)=1,2$.
If $d(w)=2$, then
\[
\left(q(G)-2\right)x_w=x_v+x_u \textrm{ and } \left(q(G)-2\right)x_u\ge x_z+x_w;
\]
if $d(w)=1$, then
\[
\left(q(G)-1\right)x_w= x_v \textrm{ and } \left(q(G)-1\right)x_u\ge x_z.
\]
In either case, we have  $x_u-x_w\ge \frac{1}{q(G)-1}\left(x_z-x_v\right)$, so
\[
\left(q(G)-n+3-\frac{1}{q(G)-1}\right)\left(x_z-x_v\right)\ge 0.
\]
Since $q(G)>n+2-\frac{4}{n+2}$, we have $x_z\ge x_v$.

Suppose next that $d(v)\le n-3$. Then
\begin{align*}
\left(q(G)-n+2\right)x_z&\ge x_v+\sum_{u\in N_{H}(v)}x_u+\sum_{u\in N_{Z_0}(w)}x_u,\\
\left(q(G)-d(v)\right)x_v&=x_z+\sum_{u\in N_{H}(v)}x_u+x_w,\\
\left(q(G)-d(w)\right)x_w&=x_v+\sum_{u\in N_{Z_0}(w)}x_u.
\end{align*}
So
\[
\left(q(G)-n+3\right)x_z\ge \left(q(G)-d(v)\right)x_v+\left(q(G)-d(w)-1\right)x_w.
\]
Since $\left(q(G)-d(w)\right)x_w\ge x_v$,  we have
\begin{align*}
\left(q(G)-n+3\right)x_z&\ge \left(q(G)-d(v)\right)x_v+\frac{q(G)-d(w)-1}{q(G)-d(w)}x_v\\
&=\left(q(G)-d(v)+\frac{q(G)-d(w)-1}{q(G)-d(w)}\right)x_v\\
&\ge \left(q(G)-d(v)\right)x_v\\
&\ge \left(q(G)-n+3\right)x_v.
\end{align*}
Since $q(G)>n+2-\frac{4}{n+2}$, we have $x_z\ge x_v$.
\end{proof}

From Claims \ref{211} and \ref{c31}, we have $H\ncong K_4$.

\begin{claim}\label{C4}
 $H$ does not contain $C_4$.
\end{claim}

\begin{proof} Suppose that $H$ contains $C_4$.
From Claim \ref{c2}, we have $N_Z(w)=N_{H}(w)$.
From Claim \ref{nok4}, we have $G[Z^*]\cong \frac{n-6}{4}K_4, K_1\cup \frac{n-7}{4}K_4, K_{1,1}\cup \frac{n-8}{4}K_4$ or $K_{1,s}^+\cup \frac{n-s-7}{4}K_4$.
As $H\ncong K_4$, $H\cong C_4$ or $C_4^+$.
From Claims \ref{c31} and \ref{211} and Lemma \ref{addedges}, we have $d(w)=4$ if $H\cong C_4$, $d(w)=2$ if $H\cong C_4^+$ and $w$ is adjacent to each vertex with degree two in $H$.
If $G[Z^*]\cong \frac{n-6}{4}K_4$, then $G\cong G_1$ if $H\cong C_4$, $G\cong G_2$ if $H\cong C_4^+$;
if $G[Z^*]\cong K_1\cup \frac{n-7}{4}K_4$,  then $G\cong G_3$ if $H\cong C_4$, $G\cong G_4$ if $H\cong C_4^+$;
if $G[Z^*]\cong K_{1,s}^+\cup \frac{n-s-7}{4}K_4$ and $s=2$, then $G\cong G_5$ if $H\cong C_4$, $G\cong G_6$ if $H\cong C_4^+$. In all such cases,
from Lemma \ref{930}, we have $q(G)<q(K_{1,1,n-2}^+)$, a contradiction.
Suppose that $G[Z^*]\cong K_{1,1}\cup \frac{n-8}{4}K_4$ or $K_{1,s}^+\cup \frac{n-s-7}{4}K_4$ where $s\ge 3$.
Note that $2e(H)+e(V(H),\{w\})=12$ and $2e(G[Z^*])=3n-22$ if $G[Z^*]\cong K_{1,1}\cup \frac{n-8}{4}K_4$, $2e(G[Z^*])=3n-s-19\le 3n-22$ if $G[Z^*]\cong K_{1,s}^+\cup \frac{n-s-7}{4}K_4$ where $s\ge 3$.
Then
\[
\eta_G(z)\le n-2+\frac{n-2+12+3n-22}{n-2}=n+2-\frac{4}{n-2}<n+2-\frac{4}{n+2}.
\]
For any $v\in V(G)\backslash \{z\}$, we have $d(v)\le n-4$.
From Claim \ref{new} and Lemmas \ref{q1} and \ref{712},  we have $q(G)\le n+2-\frac{4}{n+2}<q(K_{1,1,n-2}^+)$, also a contradiction.
\end{proof}

\begin{claim}\label{k1se}
$H\ncong K_{1,s}^+$ for any $s\ge 2$.
\end{claim}
\begin{proof} Suppose that $H\cong K_{1,s}^+$ for some $s\ge 2$.

\noindent
{\bf Case i.} $s=2$.

From Lemma \ref{addedges}, $e(W,Z)=|Z_0|+3$.

Suppose that there is a tree component in $G[Z_+]$. From Claim \ref{nok4}, there is exactly one such component $K_{1,1}$  in $G[Z_+]$.
Note that $e(Z)=4+\frac{3}{2}(n-|Z_0|-7)$.
Then
\begin{align*}
\eta_G(z)&\le n-2+\frac{n-2+8+3(n-|Z_0|-7)+|Z_0|+3}{n-2}=n+2-\frac{2|Z_0|+4}{n-2}<n+2-\frac{4}{n+2}.
\end{align*}
Note that $d(v)\le n-4$ for any $v\in V(G)\backslash\{z\}$. From Claim \ref{new} and Lemmas \ref{712} and \ref{q1} that $q(G)<q(K_{1,1,n-2}^+)$, a contradiction.

Suppose that there is no tree component in $G[Z_+]$.

Suppose first that there exists a component which is $K_{1,s'}^+$ in $G[Z^*]$ for some $2\le s'\le n-6$. Then from Claim \ref{nok4} we have $e(Z)= s'+4+\frac{3}{2}\left(n-2-|Z_0|-3-s'-1\right)=s'+4+\frac{3}{2}\left(n-|Z_0|-s'-6\right)$.
If $|Z_0|\ge 1$, then
\begin{align*}
\eta_G(z)&\le n-2+\frac{n-2+2s'+8+3n-3|Z_0|-3s'-18+|Z_0|+3}{n-2}\\
&=n+2-\frac{2|Z_0|+s'+1}{n-2}\\
&\le n+2-\frac{5}{n-2}<n+2-\frac{4}{n+2},
\end{align*}
and  $d(v)\le n-5$ for any $v\in V(G)\backslash\{z\}$.
Suppose that $Z_0=\emptyset$. If $s'=2$, then $G\cong G_{7}$, so,
from Lemma \ref{930}, we have $q(G)<q(K_{1,1,n-2}^+)$, a contradiction.
If $s'\ge 3$, it  can be obtained as above that $\eta_{G}(z)<n+2-\frac{4}{n+2}$.
Note that $d(v)\le n-5$ for any $v\in V(G)\backslash\{z\}$.
So, whether $Z_0$ is empty or not, we have from Claim \ref{new} and Lemmas \ref{712} and \ref{q1} that $q(G)<q(K_{1,1,n-2}^+)$, a contradiction.

Suppose next that there is no component which is $K_{1,s'}^+$ in $G[Z_+]$ for any $2\le s'\le n-6$. Then $e(Z)=3+\frac{3}{2}\left(n-|Z_0|-5\right)$.
If  $Z_0=\emptyset$, then $G$ is the proper subgraph of $K_1\vee \frac{n-1}{4}K_4$,
which does not contain a trebly chorded cycle with chords incident to a vertex.
From Lemma \ref{addedges}, we have $q(G)<q(K_1\vee \frac{n-1}{4}K_4)$, a contradiction.
If $|Z_0|=1$, then $G\cong G_{8}$, so
from Lemma \ref{930}, we have $q(G)<q(K_{1,1,n-2}^+)$, a contradiction.
If $|Z_0|\ge 2$, then
\begin{align*}
\eta_G(z)&=n-2+\frac{n-2+6+3n-3|Z_0|-15+3+|Z_0|}{n-2}\\
&=n+2-\frac{2|Z_0|}{n-2}\\
&\le n+2-\frac{4}{n-2}<n+2-\frac{4}{n+2},
\end{align*}
and  $d(v)\le n-4$ for any $v\in V(G)\backslash\{z\}$, so we can arrive at a contradiction as above from
Claim \ref{new} and Lemmas \ref{712} and \ref{q1}.


\noindent
{\bf Case ii.}   $3\le s\le n-3$.

From Claim \ref{211}, we have $d_{H}(w)\ge 2$, and so from Claim \ref{c31},  $w$ is not adjacent to the center of $H$.
From Lemma \ref{addedges}, we have $d_{H}(w)=3$.
From Claim \ref{c2}, we have $d_{Z_0}(w)=0$.
Note that $Z_0=\emptyset$, otherwise, let $G'$ be a graph obtained from $G$ by adding all possible edges between $Z_0$ and the center of $H$.

Suppose next that there is a tree component in $G[Z_+]$. From Claim \ref{nok4}, there is exactly one such component $K_{1,1}$ in $G[Z_+]$.
Note that $e(Z)= n-3+\frac{n-s-5}{2}$.
Then
\begin{align*}
\eta_G(z)&= n-2+\frac{n-2+2n-6+n-s-5+3}{n-2}\\
&=n+2-\frac{s+2}{n-2}\le n+2-\frac{5}{n-2}<n+2-\frac{4}{n+2}.
\end{align*}
Note that $d(v)\le n-4$ for any $v\in V(G)\backslash\{z\}$.
As above, we have a contradiction.

Suppose  that there is no tree component in $G[Z_+]$.

Suppose first that there exists another component which is $K_{1,s'}^+$ in $G[Z_+]$ for some $2\le s'\le n-6$.  Then  $e(Z)\le n-2+\frac{n-s-6}{2}$.
So
\begin{align*}
\eta_G(z)&\le n-2+\frac{n-2+2n-4+n-s-6+3}{n-2}\\
&=n+2-\frac{s+1}{n-2}\le n+2-\frac{4}{n-2}<n+2-\frac{4}{n+2}.
\end{align*}
Note that $d(v)\le n-4$ for any $v\in V(G)\backslash\{z\}$. As above, we have a contradiction.

Suppose next  that there is no component which is $K_{1,s'}^+$ in $G[Z_+]$ for any $2\le s'\le n-6$. From Claim \ref{nok4}, we have $G[Z^*]\cong \frac{n-s-3}{4}K_4$.
If $s=3,4,5$, then $G\cong G_{9}, G_{10}, G_{11}$, so, from Lemma \ref{930}, we have $q(G)<q(K_{1,1,n-2}^+)$, a contradiction.
Suppose that $s\ge 6$. Note that $e(Z)=n-2+\frac{n-s-3}{2}$.
Then
\begin{align*}
\eta_G(z)&=n-2+\frac{n-2+2n-4+n-s-3+3}{n-2}\\
&=n+2-\frac{s-2}{n-2}\le n+2-\frac{4}{n-2}<n+2-\frac{4}{n+2}.
\end{align*}
Let $v$ be the center of $H$.  If  $s=n-3$, then either
$n=7$,  $G\cong G_{10}$, and so, from Lemma \ref{930}, we have $q(G)<q(K_{1,1,n-2}^+)$, a contradiction, or $n\ge 8$, and
\[
\eta_G(v)=n-2+\frac{n-2+4\times2+3+2(n-6)}{n-2}=n+1+\frac{3}{n-2}<n+2-\frac{4}{n+2}.
\]
If $s\le n-4$,  then, as $Z_0=\emptyset$, we have $s\neq n-4$, so  $d(v)\le n-4$.
Note that $d(u)\le n-4$ for any $u\in V(G)\backslash\{z,v\}$.
As above, we have a contradiction.
\end{proof}

\begin{claim}\label{sn1n2}
$H\ncong S_{n_1,n_2}$ for any $n_1, n_2\ge 1$.
\end{claim}
\begin{proof} Suppose that $H\cong S_{n_1,n_2}$ for some $n_1, n_2$.

Suppose first that $H\cong S_{1,1}$. From Lemma \ref{p5}, we have $e(Z_+)\le 3+\frac{3}{2}\left(n-6-|Z_0|\right)$.
From Claim \ref{211}, $d_{H}(w)\ge 2$.
If  $d_{H}(w)= 2$, then $e(Z,W)\le 2+|Z_0|$, so
\begin{align*}
\eta_G(z)&\le n-2+\frac{n-2+6+3n-18-3|Z_0|+2+|Z_0|}{n-2}\\
&=n+2-\frac{2|Z_0|+4}{n-2}<n+2-\frac{4}{n+2}.
\end{align*}
Suppose that $d_{H}(w)=3,4$. Then at least one pendant vertex of $H$ is a neighbor of $w$.
From Claim \ref{c2}, we have $d_{Z_0}(w)=0$.
So $e(Z,W)\le 4$. From Lemma \ref{addedges},  we have $e(Z,W)=4$.
If $e(Z_+)=3+\frac{3}{2}\left(n-6\right)$, then $G$ is a subgraph of $G_1$,
so we have from Lemma \ref{930} that $q(G)<q(K_{1,1,n-2}^+)$, a contradiction.
If $e(Z_+)\le 2+\frac{3}{2}\left(n-6\right)$, then
\begin{align*}
\eta_G(z)&\le n-2+\frac{n-2+4+3n-18+4}{n-2}\\
&=n+2-\frac{4}{n-2}<n+2-\frac{4}{n+2}.
\end{align*}

Note that $d(v)\le n-4$ for any $v\in V(G)\backslash\{z\}$.
As above, we have a contradiction.

Suppose next that $n_1+n_2\ge 3$. From Claim \ref{c31}, we have $e(Z_+,W)\le 3$.
Then $e(Z,W)\le |Z_0|+3$.
From Lemma \ref{p5}, we have $e(Z_+)\le n_1+n_2+1+\frac{3}{2}\left(n-n_1-n_2-4-|Z_0|\right)$.
Then
\begin{align*}
\eta_G(z)&\le n-2+\frac{n-2+2n_1+2n_2+2+3n-3n_1-3n_2-12-3|Z_0|+|Z_0|+3}{n-2}\\
&=n+2-\frac{n_1+n_2+2|Z_0|+1}{n-2}\\
&\le n+2-\frac{2|Z_0|+4}{n-2}\\
&\le n+2-\frac{4}{n-2}<n+2-\frac{4}{n+2}.
\end{align*}
Let $u_1$ and $u_2$ be the centers of $S_{n_1,n_2}$.
Then $d_H(u_1)=n_1+1$ and $d_H(u_2)=n_2+1$.
If $d_H(u_2)=n-5$, then $d(u_2)=n-3$ and $wu_2\in E(G)$, so
\[
\eta_G(u_2)\le n-3+\frac{n-2+4+4+3+2(n-7)}{n-3}=n+\frac{4}{n-3}<n+2-\frac{4}{n+2}.
\]
Suppose that $d_H(u_2)=n-4$, then $d(u_2)=n-3$ or $n-2$.
If $d(u_2)=n-3$, then  $wu_2\notin E(G)$, so
\[
\eta_G(u_2)\le n-3+\frac{n-2+4+3+2(n-6)}{n-3}=n+\frac{2}{n-3}<n+2-\frac{4}{n+2}.
\]
If $d(u_2)=n-2$, then $wu_2\in E(G)$, so
\[
\eta_G(u_2)\le n-2+\frac{n-2+4+3+2(n-6)+3}{n-2}=n+1+\frac{2}{n-2}<n+2-\frac{4}{n+2}.
\]
Note that $d(v)\le n-4$ for any $v\in V(G)\backslash\{z,u_2\}$.
As above, we have a contradiction.
\end{proof}

\begin{claim}\label{k1s}
$H\ncong K_{1,s}$ for any $s\ge 1$.
\end{claim}
\begin{proof} Suppose that $H\cong K_{1,s}$ for some $s$.
From Lemma \ref{p5}, we have $e(Z_+)\le \frac{3}{2}\left(n-2-s-1-|Z_0|\right)+s=\frac{3}{2}\left(n-s-|Z_0|-3\right)+s$.

Suppose first that $s=1$.
If $Z_0=\emptyset$ and $e(Z_+)=\frac{3}{2}\left(n-4\right)+1$, then $G$ is the proper subgraph of $K_1\vee \left(K_3\cup \frac{n-4}{4}K_4\right)$, which  does not contain a trebly chorded cycle with chords incident to a vertex.
From Lemma \ref{addedges}, we have $q(G)<q\left(K_1\vee \left(K_3\cup \frac{n-4}{4}K_4\right)\right)$, a contradiction.
If $Z_0=\emptyset$ and $e(Z_+)\le\frac{3}{2}\left(n-4\right)$, then
\[
\eta_{G}(z)\le n-2+\frac{n-2+3n-12+2}{n-2}=n+2-\frac{4}{n-2}<n+2-\frac{4}{n+2};
\]
if $|Z_0|\ge 1$, then
\begin{align*}
\eta_G(z)&\le n-2+\frac{n-2+3n-3-3|Z_0|-9+2+|Z_0|+2}{n-2}\\
&=n+2-\frac{2|Z_0|+2}{n-2}\le n+2-\frac{4}{n-2}<n+2-\frac{4}{n+2}.
\end{align*}
If there is no other component in $G[Z_+]$, then from Lemma \ref{addedges} we have $N(z)=N(w)$. So $\eta_G(w)=\eta_G(z)<n+2-\frac{4}{n+2}$. Otherwise, $d(w)\le n-4$.
Note that $d(v)\le n-4$ for any $v\in V(G)\backslash\{z\}$.
As above, we have a contradiction.

Suppose next that $s=2$.
If $Z_0=\emptyset$ and $e(Z_+)=\frac{3}{2}\left(n-5\right)+2$, then $G$ is the proper subgraph of $K_1\vee \left(C_4^+\cup \frac{n-5}{4}K_4\right)$, which does not contain a trebly chorded cycle with chords incident to a vertex.
From Lemma \ref{addedges}, we have $q(G)<q\left(K_1\vee \left(C_4^+\cup \frac{n-5}{4}K_4\right)\right)$, a contradiction.
If $Z_0=\emptyset$ and $e(Z_+)\le\frac{3}{2}\left(n-5\right)+1$, then
\[
\eta_{G}(z)\le n-2+\frac{n-2+3n-15+2+3}{n-2}=n+2-\frac{4}{n-2}<n+2-\frac{4}{n+2};
\]
if $|Z_0|\ge 1$, then
\begin{align*}
\eta_G(z)&\le n-2+\frac{n-2+3n-6-3|Z_0|-9+4+|Z_0|+3}{n-2}\\
&=n+2-\frac{2|Z_0|+2}{n-2}\le n+2-\frac{4}{n-2}<n+2-\frac{4}{n+2}.
\end{align*}
If there is no other component in $G[Z_+]$, then from Lemma \ref{addedges} we have $N(z)=N(w)$. So $\eta_G(w)=\eta_G(z)<n+2-\frac{4}{n+2}$. Otherwise, $d(w)\le n-4$.
Note that $d(v)\le n-4$ for any $v\in V(G)\backslash\{z\}$.
As above, we have a contradiction.

Suppose finally that $s\ge 3$.
From Claim \ref{c31}, we have $e(Z,W)\le s+|Z_0|$.
Suppose that $Z_0=\emptyset$.  If $e(Z_+)=\frac{3}{2}\left(n-s-3\right)+s$, then $G\cong G_{12}$, so
from  Lemma \ref{930}, we have $q(G)<q(K_{1,1,n-2}^+)$, a contradiction.
If $e(Z_+)\le\frac{3}{2}\left(n-s-3\right)+s-1$, then
\[
\eta_{G}(z)\le n-2+\frac{n-2+3n-3s-9+2s-2+s}{n-2}=n+2-\frac{5}{n-2}<n+2-\frac{4}{n+2}.
\]
 Evidently, $s\le n-3$.
If $s=n-3$, then letting $u$ be the center of $K_{1,s}$, we have $wu\notin E(G)$, so $G\cong G_{13}$, and from  Lemma \ref{930}, we have $q(G)<q(K_{1,1,n-2}^+)$, a contradiction. If $s<n-3$,
then  $d(v)\le n-4$ for any $v\in V(G)\backslash\{z\}$.
As above, we have a contradiction.

Suppose that $|Z_0|\ge 1$. Then
\begin{align*}
\eta_G(z)&\le n-2+\frac{n-2+3n-3s-3|Z_0|-9+2s+|Z_0|+s}{n-2}\\
&=n+2-\frac{2|Z_0|+3}{n-2}\le n+2-\frac{5}{n-2}<n+2-\frac{4}{n+2}
\end{align*}
Note that $s\le n-4$. If $s=n-4$,
then, letting $u$ be the center of $K_{1,s}$, we have
\[
\eta_G(u)\le n-3+\frac{n-2+3(n-4)}{n-3}=n+1-\frac{2}{n-3}<n+2-\frac{4}{n+2}
\]
when $wu\notin E(G)$, and from Claim \ref{c31},  $d(w)\le 2$, so
\[
\eta_G(u)\le n-2+\frac{n-2+2(n-4)+2}{n-2}=n+1-\frac{2}{n-2}<n+2-\frac{4}{n+2}.
\]
when $wu\in E(G)$.
If $d(w)=n-3$, then $d(v)\le 3$ for any $v\in N(w)$, so $\eta_G(w)\le n-3+3=n<n+2-\frac{4}{n+2}$. 
In the remaining case, 
 $d(v)\le n-4$ for any $v\in V(G)\backslash\{z\}$.
As above, we have $q(G)<q(K_{1,1,n-2}^+)$, a contradiction.
\end{proof}

From Claims \ref{C4}-\ref{k1s},
there is no component of $G[Z_+]$ contains the neighbors of $w$, which is contradict to Claim \ref{6662}.
\end{proof}

\section{Concluding remarks}

Note $K_{1,1,n-2}^+$ does not contain any trebly chorded cycle.
From Theorem \ref{doubly21}, we have

\begin{corollary} \label{RE}
Suppose that $G$ is  a graph of order  $n$ without isolated vertices, where $n\ge 7$.
If $G$ does not contain a trebly chorded cycle, then
\[
q(G)\le q(K_{1,1,n-2}^+)
\]
unless $G\cong K_{1,1,n-2}^+$.
\end{corollary}

It is interesting to find  signless Laplacian index conditions that imply a graph with fixed order
contains a cycle with $2k+1$ chords incident to a vertex on the cycle for $k\ge 2$.

\bigskip
\bigskip

\noindent {\bf Acknowledgement.}
This work was supported by the National Natural Science Foundation of China (No.~12571364).

\begin{appendices}

\section{Proof of Lemma \ref{930}}

1. Proof of $q(G_1)<q(K_{1,1,n-2}^+)$:

Note that $Q(G_1)$ has an equitable quotient matrix
$B_1=\left(
\begin{smallmatrix}
n-2 &  4  & n-6 & 0 \\
1   &  6  & 0   & 1 \\
1   &  0  & 7   & 0 \\
0   &  4  & 0   & 4
\end{smallmatrix}\right)$
with CP
\[
g_1(x)=x^4 - ( n + 15)x^3 + (16n + 58)x^2 -(80n-24 )x + 120n - 272.
\]
Then
\begin{align*}
g_1'(x)&=4x^3-(3n + 45)x^2 + (32n + 116)x - 80n + 24,\\
g_1''(x)&=12x^2 - ( 6n+ 90)x + 32n + 116.
\end{align*}
For $x\ge n+1$, as $n\ge 10$, we have
\begin{align*}
g_1''(x)&\ge g_1''(n+1)=6n^2 - 40n + 38>0,\\
g_1'(x)&\ge g_1'(n+1)=n^3 - 7n^2 - 13n + 99>0,
\end{align*}
so
\[
g_1(x)\ge g_1(n+1)=2n^3 - 32n^2 + 154n - 204>0.
\]
From Lemmas \ref{quo} and \ref{712}, we have $q(G_1)<n+1<n+2-\frac{4}{n+2}<q(K_{1,1,n-2}^+)$.\\

\noindent
2. Proof of $q(G_2)<q(K_{1,1,n-2}^+)$:

Note that $Q(G_2)$ has an equitable quotient matrix
$B_2=
\left(\begin{smallmatrix}
n-2 &  3  & 4 & n-9 & 0 \\
1   &  5  & 0 & 0   & 0 \\
1   &  0  & 6 & 0   & 1 \\
1   &  0  & 0 & 7   & 0 \\
0   &  0  & 4 & 0   & 4
\end{smallmatrix}\right)$
with CP
\[
g_2(x)=x^5-(n+20)x^4+(21n+133)x^3-(160n+260)x^2+(520n-452)x-600n+1480.
\]
Then
\begin{align*}
g_2'(x)&=5x^4 - (4n + 80)x^3 + (63n + 399)x^2 - ( 320n + 520)x + 520n - 452,\\
g_2''(x)&=20x^3 - ( 12n + 240)x^2 + (126n + 798)x - 320n - 520,\\
g_2^{(3)}(x)&=60x^2 - ( 24n + 480)x + 126n + 798.
\end{align*}
For $x\ge n+1$, as $n\ge 9$, we have
\begin{align*}
g_2^{(3)}(x)&\ge g_2^{(3)}(n+1)=36n^2 - 258n + 378>0,\\
g_2''(x)&\ge g_2''(n+1)=8n^3 - 78n^2 + 172n + 58>0,\\
g_2'(x)&\ge g_2'(n+1)=n^4 - 9n^3 - 17n^2 + 317n - 648>0,
\end{align*}
so
\[
g_2(x)\ge g_2(n+1)=2n^4 - 40n^3 + 288n^2 - 868n + 882>0.
\]
From Lemmas \ref{quo} and \ref{712}, we have $q(G_2)<n+1<n+2-\frac{4}{n+2}<q(K_{1,1,n-2}^+)$.\\

\noindent
3. Proof of $q(G_3)<q(K_{1,1,n-2}^+)$:

Note that $Q(G_3)$ has an equitable quotient matrix
$B_3=
\left(\begin{smallmatrix}
n-2 &  2  & 2  & n-6 & 0 \\
1   &  5  & 2  & 0   & 0 \\
1   &  2  & 4  & 0   & 1 \\
1   &  0  & 0  & 7   & 0 \\
0   &  0  & 2  & 0   & 2
\end{smallmatrix}\right)$
with CP
\[
g_3(x)=x^5 - ( n  +16)x^4 + (17n + 75)x^3 - ( 98n +48)x^2 + (214n - 288)x - 132n + 288.
\]
Then
\begin{align*}
g_3'(x)&=5x^4 - ( 4n + 64)x^3 + (51n + 225)x^2 - ( 196n + 84)x + 214n - 348,\\
g_3''(x)&=20x^3 - (12n + 192)x^2 + (102n + 450)x - 196n - 84,\\
g_3^{(3)}(x)&=60x^2 - (24n + 384)x + 102n + 450.
\end{align*}
For $x\ge n+1$ and $n\ge 7$, we have
\begin{align*}
g_3^{(3)}(x)&\ge g_3^{(3)}(n+1)=36n^2 - 186n + 126>0,\\
g_3''(x)&\ge g_3''(n+1)=8n^3 - 54n^2 + 20n + 194>0,\\
g_3'(x)&\ge g_3'(n+1)=n^4 - 5n^3 - 43n^2 + 259n - 266>0,
\end{align*}
so
\[
g_3(x)\ge g_3(n+1)=2n^4 - 32n^3 + 162n^2 - 266n + 90>0.
\]
From Lemmas \ref{quo} and \ref{712}, we have $q(G_3)<n+1<n+2-\frac{4}{n+2}<q(K_{1,1,n-2}^+)$.\\

\noindent
4. Proof of  $q(G_4)<q(K_{1,1,n-2}^+)$:

Note that $Q(G_4)$ has an equitable quotient matrix
$B_4=
\left(\begin{smallmatrix}
n-2 &  1  & 2 & 2 & n-7 & 0 \\
1   &  1  & 0 & 0 & 0   & 0 \\
1   &  0  & 5 & 2 & 0   & 0 \\
1   &  0  & 2 & 4 & 0   & 1 \\
1   &  0  & 0 & 0 & 7   & 0 \\
0   &  0  & 0 & 2 & 0   & 2
\end{smallmatrix}\right)$
with CP
\[
g_4(x)=(x-1)g_3(x).
\]
For $x\ge n+1$, as $n\ge 7$, we have $g_4(x)\ge g_4(n+1)>0$.
From Lemmas \ref{quo} and \ref{712}, we have $q(G_4)<n+1<n+2-\frac{4}{n+2}<q(K_{1,1,n-2}^+)$.\\

\noindent
5. Proof of $q(G_5)<q(K_{1,1,n-2}^+)$:

Note that $Q(G_5)$ has an equitable quotient matrix
$B_5=
\left(\begin{smallmatrix}
n-2 &  1  & 4 & n-7 & 0 \\
1   &  1  & 0 & 0   & 0 \\
1   &  0  & 6 & 0   & 1 \\
1   &  0  & 0 & 7   & 0 \\
0   &  0  & 4 & 0   & 4
\end{smallmatrix}\right)$
with CP
\[
g_5(x)=x^5 -(n +16)x^4 + (17n + 73)x^3 - (96n +28)x^2 + (200n - 356)x - 120n + 392.
\]
Then
\begin{align*}
g_5'(x)&=5x^4 - ( 4n + 64)x^3 + (51n + 219)x^2 -( 192n + 56)x + 200n - 356,\\
g_5''(x)&=20x^3 - ( 12n + 192)x^2 + (102n + 438)x - 192n - 56,\\
g_5^{(3)}(x)&=60*x^2 -( 24n + 384)x + 102n + 438.
\end{align*}
For $x\ge n+1$, as $n\ge 7$, we have
\begin{align*}
g_5^{(3)}(x)&\ge g_5^{(3)}(n+1)=36n^2 - 186n + 114>0,\\
g_5''(x)&\ge g_5''(n+1)=8n^3 - 54n^2 + 12n + 210>0,\\
g_5'(x)&\ge g_5'(n+1)=n^4 - 5n^3 - 45n^2 + 265n - 252>0,
\end{align*}
so
\[
g_5(x)\ge g_5(n+1)=2n^4 - 32n^3 + 160n^2 - 252n + 66>0.
\]
From Lemmas \ref{quo} and \ref{712}, we have $q(G_5)<n+1<n+2-\frac{4}{n+2}<q(K_{1,1,n-2}^+)$.\\

\noindent
6. Proof of  $q(G_6)<q(K_{1,1,n-2}^+)$:

Note that  $Q(G_6)$ has an equitable quotient matrix
$B_6=
\left(\begin{smallmatrix}
n-2  & 3 & 2 & 2 & n-9 & 0\\
1  &5 &0 &0 &0 &0\\
1  &0 &4 &2 &0 &1\\
1  &0 &2 &5 &0 &0\\
1  &0 &0 &0 &7 &0\\
0  &0 &2 &0 &0 &2
\end{smallmatrix}\right)$
with CP
\begin{align*}
g_6(x)&=x^6 -( n + 21)x^5 + (22n + 155)x^4 - ( 183n + 417)x^3 + (704n - 114)x^2\\
& \quad + (1920 - 1202n)x + 660n - 1572.
\end{align*}
Then
\begin{align*}
g_6'(x)&=6x^5 -( 5n + 105)x^4 + (88n + 620)x^3 - ( 549n + 1251)x^2 + (1408n - 228)x - 1202n + 1920,\\
g_6''(x)&=30x^4 - ( 20n + 420)x^3 + (264n + 1860)x^2 - ( 1098n + 2502)x + 1408n - 228,\\
g_6^{(3)}(x)&=120x^3 - ( 60n + 1260)x^2 + (528n + 3720)x - 1098n - 2502,\\
g_6^{(4)}(x)&=360x^2 - ( 120n + 2520)x + 528n + 3720.
\end{align*}
For  $x\ge n+1$, as $n\ge 9$, we have
\begin{align*}
g_6^{(4)}(x)&\ge g_6^{(4)}(n+1)=240n^2 - 1392n + 1560>0,\\
g_6^{(3)}(x)&\ge g_6^{(3)}(n+1)=60n^3 - 492n^2 + 930n + 78>0,\\
g_6''(x)&\ge g_6''(n+1)=10n^4 - 96n^3 + 150n^2 + 632n - 1260>0,\\
g_6'(x)&\ge g_6'(n+1)=n^5 - 7n^4 - 55n^3 + 593n^2 - 1520n + 962>0,
\end{align*}
so
\[
g_6(x)\ge g_6(n+1)=2n^5 - 40n^4 + 290n^3 - 890n^2 + 962n - 48>0.
\]
From Lemmas \ref{quo} and \ref{712}, we have $q(G_6)<n+1<n+2-\frac{4}{n+2}<q(K_{1,1,n-2}^+)$.\\

\noindent
7. Proof of  $q(G_{7})<q(K_{1,1,n-2}^+)$:

Note that $Q(G_{7})$ has an equitable quotient matrix
$B_{7}=
\left(\begin{smallmatrix}
n-2 & 3 & 3 & n-8 & 0\\
1 &6 &0 &0 &1\\
1 &0 &5 &0 &0\\
1 &0 &0 &7 &0\\
0 &3 &0 &0 &3
\end{smallmatrix}\right)$
with CP
\[
g_{7}(x)=(x-6)\left(x^4 - ( n + 13)x^3 + (14n + 40)x^2 + (42 - 60n)x + 75n - 180\right).
\]
Let  $f(x)=\frac{g_7(x)}{x-6}$.
Then
\begin{align*}
f'(x)&=4x^3 - (3n + 39)x^2 + (28n + 80)x - 60n + 42,\\
f''(x)&=12x^2 - ( 6n + 78)x + 28n + 80.
\end{align*}
For  $x\ge n+1$, as $n\ge 8$, we have
\begin{align*}
f''(x)&\ge f''(n+1)=6n^2 - 32n + 14>0,\\
f'(x)&\ge f'(n+1)=n^3 - 5n^2 - 21n + 87>0,
\end{align*}
so
\[
f(x)\ge f(n+1)=2n^3 - 28n^2 + 115n - 110>0.
\]
From Lemmas \ref{quo} and \ref{712}, we have $q(G_{7})<n+1<n+2-\frac{4}{n+2}<q(K_{1,1,n-2}^+)$.\\

\noindent
8. Proof of $q(G_{8})<q(K_{1,1,n-2}^+)$:

Note that $Q(G_{8})$ has an equitable quotient matrix
$
B_{8}=
\left(\begin{smallmatrix}
n-2 & 1 & 3 & n-6 & 0\\
1 &2 &0 &0 &1\\
1 &0 &6 &0 &1\\
1 &0 &0 &7 &0\\
0 &1 &3 &0 &4
\end{smallmatrix}\right)$
with CP
\[
g_{8}(x)=x^5 -( n + 17)x^4 + (18n + 88)x^3 - ( 112n + 84)x^2 + (276n - 384)x - 216n + 624.
\]
Then
\begin{align*}
g_{8}'(x)&=5x^4 - ( 4n + 68)x^3 + (54n + 264)x^2 - ( 224n + 168)x + 276n - 384,\\
g_{8}''(x)&=20x^3 - ( 12n + 204)x^2 + (108n + 528)x - 224n - 168,\\
g_{8}^{(3)}(x)&=60x^2 - ( 24n + 408)x + 108n + 528.
\end{align*}
For  $x\ge n+1$, as $n\ge 10$, we have
\begin{align*}
g_{8}^{(3)}(x)&\ge g_{8}^{(3)}(n+1)=36n^2 - 204n + 180>0,\\
g_{8}''(x)&\ge g_{8}''(n+1)=8n^3 - 60n^2 + 52n + 176>0,\\
g_{8}'(x)&\ge g_{8}'(n+1)=n^4 - 6n^3 - 38n^2 + 278n - 351>0,
\end{align*}
so
\[
g_{8}(x)\ge g_{8}(n+1)=2n^4 - 34n^3 + 190n^2 - 386n + 228>0.
\]
From Lemmas \ref{quo} and \ref{712}, we have $q(G_{8})<n+1<n+2-\frac{4}{n+2}<q(K_{1,1,n-2}^+)$.\\

\noindent
9. Proof of $q(G_{9})<q(K_{1,1,n-2}^+)$:

Note that  $Q(G_{9})$ has an equitable quotient matrix
$B_{9}=
\left(\begin{smallmatrix}
n-2 & 2 & 1 & 1 & n-6 & 0\\
1 &5 &1 &0 &0 &1\\
1 &2 &4 &1 &0 &0\\
1 &0 &1 &3 &0 &1\\
1 &0 &0 &0 &7 &0\\
0 &2 &0 &1 &0 &3
\end{smallmatrix}\right)$
with CP
\begin{align*}
g_{9}(x)&=x^6-( n + 20)x^5 + (21n + 140)x^4 - ( 167n + 358)x^3 + (620n -91)x^2\\
&\quad + (1612 - 1051n)x + 618n - 1524.
\end{align*}
Then
\begin{align*}
g_{9}'(x)&=6x^5 - (5n + 100)x^4 + (84n + 560)x^3 - ( 501n +1074)x^2+ (1240n - 182)x\\
& - 1051n + 1612,\\
g_{9}''(x)&=30x^4 - ( 20n + 400)x^3 + (252n + 1680)x^2 - ( 1002n + 2148)x + 1240n - 182,\\
g_{9}^{(3)}(x)&=120x^3 - ( 60n + 1200)x^2 + (504n + 3360)x - 1002n - 2148,\\
g_{9}^{(4)}(x)&=360x^2 - ( 120n + 2400)x + 504n + 3360.
\end{align*}
For  $x\ge n+1$, as $n\ge 10$, we have
\begin{align*}
g_{9}^{(4)}(x)&\ge g_{9}^{(4)}(n+1)=240n^2 - 1296n + 1320>0,\\
g_{9}^{(3)}(x)&\ge g_{9}^{(3)}(n+1)=60n^3 - 456n^2 + 762n + 132>0,\\
g_{9}''(x)&\ge g_{9}''(n+1)=10n^4 - 88n^3 + 102n^2 + 602n - 1020>0,\\
g_{9}'(x)&\ge g_{9}'(n+1)=n^5 - 6n^4 - 59n^3 + 536n^2 - 1253n + 822>0,
\end{align*}
so
\[
g_{9}(x)\ge g_{9}(n+1)=2n^5 - 38n^4 + 257n^3 - 743n^2 + 862n - 240>0.
\]
From Lemmas \ref{quo} and \ref{712}, we have $q(G_{9})<n+1<n+2-\frac{4}{n+2}<q(K_{1,1,n-2}^+)$.\\

\noindent
10. Proof of $q(G_{10})<q(K_{1,1,n-2}^+)$:

For $n=7$,  $q(G_{10})=8.2351<8.7355=q(K_{1,1,n-2}^+)$.
Suppose $n\ge 11$. Note that $Q(G_{10})$ has an equitable quotient matrix
$
B_{10}=
\left(\begin{smallmatrix}
n-2 & 2 & 1& 1& 1 & n-7 & 0\\
1 &5 &1 & 0&0 &0 &1\\
1 &2 &5 & 1&1 &0 &0\\
1 &0 &1 & 2 &0 &0 &0\\
1 &0 &1 &0 &3 &0 &1\\
1 &0 &0 & 0&0 &7 &0\\
0 &2 &0 &0 &1 &0 &3
\end{smallmatrix}\right)$
with CP
\begin{align*}
g_{10}(x)&=x^7 -( n + 23)x^6 + (24n + 197)x^5 -( 227n + 731)x^4 + (1073n + 739)x^3 \\
&\quad + (2307 - 2643n)x^2 + (3174n - 6322)x - 1440n + 4256.
\end{align*}
Then
\begin{align*}
g_{10}'(x)&=7x^6 -( 6n + 138)x^5 + (120n + 985)x^4 -( 908n + 2924)x^3+ (3219n + 2217)x^2 \\
&\quad + (4614 - 5286n)x + 3174n - 6322,\\
g_{10}''(x)&=42x^5 - ( 30n + 690)x^4 + (480n + 3940)x^3 - ( 2724n +8772)x^2+ (6438n + 4434)x\\
& \quad - 5286n + 4614,\\
g_{10}^{(3)}(x)&=210x^4 -( 120n + 2760)x^3 + (1440n + 11820)x^2 - ( 5448n + 17544)x + 6438n + 4434,\\
g_{10}^{(4)}(x)&=840x^3 -( 360n + 8280)x^2 + (2880n + 23640)x - 5448n - 17544,\\
g_{10}^{(5)}(x)&=2520x^2 - ( 720n +16560)x + 2880n + 23640.
\end{align*}
For  $x\ge n+1$, as $n\ge 11$, we have
\begin{align*}
g_{10}^{(5)}(x)&\ge g_{10}^{(5)}(n+1)=1800n^2 - 9360n + 9600>0,\\
g_{10}^{(4)}(x)&\ge g_{10}^{(4)}(n+1)=480n^3 - 3600n^2 + 6672n - 1344>0,\\
g_{10}^{(3)}(x)&\ge g_{10}^{(3)}(n+1)=90n^4 - 840n^3 + 1872n^2 + 966n - 3840>0,\\
g_{10}''(x)&\ge g_{10}''(n+1)=12n^5 - 120n^4 + 136n^3 + 1638n^2 - 4962n + 3568>0,\\
g_{10}'(x)&\ge g_{10}'(n+1)=n^6 - 6n^5 - 88n^4 + 931n^3 - 3042n^2 + 3881n - 1561>0,
\end{align*}
so
\[
g_{10}(x)\ge g_{10}(n+1)=2n^6 - 42n^5 + 329n^4 - 1201n^3 + 2097n^2 - 1601n + 424>0.
\]
From Lemmas \ref{quo} and \ref{712}, we have $q(G_{10})<n+1<n+2-\frac{4}{n+2}<q(K_{1,1,n-2}^+)$.\\

\noindent
11. Proof of $q(G_{11})<q(K_{1,1,n-2}^+)$:

For $n=8$, $q(G_{11})=9.0478<9.7324=q(K_{1,1,n-2}^+)$.
Suppose that $n\ge 12$. Note that  $Q(G_{11})$ has an equitable quotient matrix
$
B_{11}=
\left(\begin{smallmatrix}
n-2 & 2 & 1& 2& 1 & n-8 & 0\\
1 &5 &1 & 0&0 &0 &1\\
1 &2 &6 & 2&1 &0 &0\\
1 &0 &1 & 2 &0 &0 &0\\
1 &0 &1 &0 &3 &0 &1\\
1 &0 &0 & 0&0 &7 &0\\
0 &2 &0 &0 &1 &0 &3
\end{smallmatrix}\right)$
with CP
\begin{align*}
g_{11}(x)&=x^7 - ( n + 24)x^6 + (25n + 214)x^5 - ( 245n + 824)x^4 + (1192n + 841)x^3\\
& \quad + (2952 - 2995n)x^2 + (3628n - 8328)x - 1644n + 5872.
\end{align*}
Then
\begin{align*}
g_{11}'(x)&=7x^6 - ( 6n + 144)x^5 + (125n + 1070)x^4 - ( 980n + 3296)x^3 + (3576n + 2523)x^2\\
& \quad + (5904 - 5990n)x + 3628n - 8328,\\
g_{11}''(x)&=42x^5 - ( 30n + 720)x^4 + (500n + 4280)x^3 - ( 2940n + 9888)x^2 + (7152n + 5046)x\\
& \quad - 5990n + 5904,\\
g_{11}^{(3)}(x)&=210x^4 -( 120n + 2880)x^3 + (1500n + 12840)x^2 - ( 5880n +19776)x + 7152n + 5046,\\
g_{11}^{(4)}(x)&=840x^3 - ( 360n + 8640)x^2 + (3000n + 25680)x - 5880n - 19776,\\
g_{11}^{(5)}(x)&=2520x^2 -( 720n + 17280)x + 3000n + 25680.
\end{align*}
For $x\ge n+1$, as  $n\ge 12$, we have
\begin{align*}
g_{11}^{(5)}(x)&\ge g_{11}^{(5)}(n+1)=1800n^2 - 9960n + 10920>0,\\
g_{11}^{(4)}(x)&\ge g_{11}^{(4)}(n+1)=480n^3 - 3840n^2 + 7680n - 1896>0,\\
g_{11}^{(3)}(x)&\ge g_{11}^{(3)}(n+1)=90n^4 - 900n^3 + 2220n^2 + 756n - 4560>0,\\
g_{11}''(x)&\ge g_{11}''(n+1)=12n^5 - 130n^4 + 200n^3 + 1704n^2 - 5868n + 4664>0,\\
g_{11}'(x)&\ge g_{11}'(n+1)=n^6 - 7n^5 - 85n^4 + 1010n^3 - 3588n^2 + 5017n - 2264>0,
\end{align*}
so
\[
g_{11}(x)\ge g_{11}(n+1)=2n^6 - 44n^5 + 363n^4 - 1414n^3 + 2685n^2 - 2304n + 704>0.
\]
From Lemmas \ref{quo} and \ref{712}, we have $q(G_{11})<n+1<n+2-\frac{4}{n+2}<q(K_{1,1,n-2}^+)$.\\

\noindent
12. Proof of $q(G_{12})=q(G_{12}(n,s))<q(G_{13})<q(K_{1,1,n-2}^+)$ where $s\ge 3$ and $n\ge s+3$:

Note that  $Q(G_{12}(n,s))$ has an equitable quotient matrix
$
B_{12}=
\left(\begin{smallmatrix}
n-2 & s & 1&  n-s-3 & 0\\
1 &3 & 1 & 0 &1\\
1 & s & s+1 & 0 & 0\\
1 & 0 & 0 &7 &0\\
0 &s  &0 &0 &s
\end{smallmatrix}\right)$
with CP
\begin{align*}
g_{12}(x,s)&=x^5 - ( n +2s + 9)x^4 + (s^2 + (2n + 15)s + 10n + 11)x^3 - (( n + 6)s^2+ (17n+2)s\\
&\quad +27n- 39)x^2+ ((7n - 14)s^2 + (32n - 44)s +18n -54)x+ 6s^3 \\
&\quad - (6n-2)s^2- ( 12n-36)s.
\end{align*}
Let $h_1(x)=g_{12}(x,s)-g_{12}(x,s+4)$. Then
\begin{align*}
h_1(x)&=8x^4 - ( 8s + 8n + 76)x^3 + ((8n + 48)s + 84n + 104)x^2-((56n-112)s + 240n- 400)x\\
&\quad - 72s^2 + (48n - 304)s + 144n - 560,\\
h_1'(x)&=32x^3 - ( 24n + 24s + 228)x^2 + ((16n + 96)s + 168n + 208)x \\
&\quad -(56n-112)s- 240n + 400,\\
h_1''(x)&=96x^2 - ( 48n +48s + 456)x + (16n + 96)s + 168n + 208.
\end{align*}
For  $x\ge n+1$, as $n\ge s+7$, we have
\begin{align*}
h_1''(x)\ge h_1''(n+1)&=48n^2 - ( 32s + 144)n + 48s - 152\ge 16s^2 + 352s + 1192>0,\\
h_1'(x)\ge h_1'(n+1)&=8n^3 - ( 8s + 12)n^2 + (8s - 248)n + 184s + 412\\
&\ge  52s^2 + 608s + 832>0,
\end{align*}
so
\begin{align*}
h_1(x)\ge h_1(n+1)&=16n^3 - ( 16s + 172)n^2 + (184s + 392)n - 72s^2 - 152s -124\\
&\ge  52s^2 + 688s - 320>0.
\end{align*}
From Lemma \ref{quo}, $q(G_{12}(n,s))<q(G_{12}(n,s+4))$.
Note that $G_{12}(n,n-3)\cong G_{13}$.
So $q(G_{13})>q(G_{12}(n,s))$.

Note that $Q(G_{13})$ has an equitable quotient matrix
$
B_{13}=
\left(\begin{smallmatrix}
n-2 & n-3 & 1&  0\\
1 &3 & 1 &1\\
1 & n-3 & n-2  & 0\\
0 &n-3  &0 &n-3
\end{smallmatrix}\right)$
with CP
\begin{align*}
g_{13}(x)&=x^4 - (3n-4)x^3 + (3n^2 - 8n + 3)x^2 -( n^3 - 4n^2 + 7n - 12)x + 4n^2 - 24n + 36.
\end{align*}
Then
\begin{align*}
g_{13}'(x)&=4x^3 - (9n-12 )x^2 + (6n^2 - 16n + 6)x - n^3 + 4n^2 - 7n + 12,\\
g_{13}''(x)&=12x^2 - (18n-24)x + 6n^2 - 16n + 6.
\end{align*}
For $x\ge n+1$, as $n\ge 6$, we have
\begin{align*}
g_{13}''(x)&\ge g_{13}''(n+1)=14n + 42>0,\\
g_{13}'(x)&\ge g_{13}'(n+1)=10n + 34>0.
\end{align*}
Since $n+2-\frac{4}{n+2}>n+1$, we have
\[
g_{13}(x)\ge g_{13}\left(n+2-\frac{4}{n+2}\right)=\frac{10n^5 + 92n^4 + 296n^3 + 800n^2 + 1152n + 576}{n^4 + 8n^3 + 24n^2 + 32n + 16}>0.
\]
From Lemmas \ref{quo} and \ref{712}, we have $q(G_{13})<n+2-\frac{4}{n+2}<q(K_{1,1,n-2}^+)$.

\section{Proof of Lemma \ref{931}}

1. Proof of $q(K_1\vee \frac{n-1}{4}K_4)<q(K_{1,1,n-2}^+)$:

Let $G=K_1\vee \frac{n-1}{4}K_4$.
If $n=9$, then $q(G)=10.3723<10.7381=q(K_{1,1,n-2}^+)$.
Suppose that $n\ge 13$. Note that $Q(G)$ has an equitable quotient matrix
$
B_{14}=
\left(\begin{smallmatrix}
n-1 &   n-1   \\
1  &  7
\end{smallmatrix}\right)$
with CP
\[
g_{14}(x)=x^2-( n + 6)x + 6n - 6.
\]
For $x\ge n+1$, we have
$g_{14}(x)\ge g_{14}(n+1)=n-11>0$.
From Lemmas \ref{quo} and \ref{712}, we have $q(G)<n+1<n+2-\frac{4}{n+2}<q(K_{1,1,n-2}^+)$.

\noindent
2. Proof of $q(K_1\vee (K_1\cup \frac{n-2}{4}K_4))<q(K_{1,1,n-2}^+)$:

Let $G= K_1\vee (K_1\cup \frac{n-2}{4}K_4)$.
If $n=10$, then $q(G)=11.0666<11.7474=q(K_{1,1,n-2}^+)$.
Suppose that $n\ge 14$. Note that $Q(G)$ has an equitable quotient matrix
$
B_{15}=
\left(\begin{smallmatrix}
n-1 &   1  & n-2   \\
1  &  1 &  0\\
1  & 0 & 7
\end{smallmatrix}\right)$
with CP
\[
g_{15}(x)=x^3 -( n + 7)x^2 + 7nx - 6n + 12.
\]
Then
\[
g_{15}'(x)=3x^2 - ( 2n + 14)x + 7n.
\]
For  $x\ge n+1$, as $n\ge 14$, we have
\begin{align*}
g_{15}'(x)\ge g_{15}'(n+1)=n^2 - 3n - 11>0.
\end{align*}
Then
\[
g_{15}(x)\ge g_{15}(n+1)=n^2 - 11n + 6>0.
\]
From Lemmas \ref{quo} and \ref{712}, we have $q(G)<n+1<n+2-\frac{4}{n+2}<q(K_{1,1,n-2}^+)$.

\noindent
3. Proof of $q(K_1\vee \left(K_{1,1}\cup \frac{n-3}{4}K_4\right))<q(K_{1,1,n-2}^+)$:

Let $G= K_1\vee \left(K_{1,1}\cup \frac{n-3}{4}K_4\right)$.
For $n=7$,  $q(G)=8.7016<8.7355=q(K_{1,1,n-2}^+)$.
Suppose that $n\ge 11$.
Note that $Q(G)$ has an equitable quotient matrix
$
B_{16}=
\left(\begin{smallmatrix}
n-1 &   2& n-3   \\
1  &  3 & 0\\
1 &  0 & 7
\end{smallmatrix}\right)$
with CP
\[
g_{16}(x)=x^3 - ( n + 9)x^2 + (9n + 12)x - 18n + 26.
\]
We have
\[
g_{16}'(x)=3x^2 - ( 2n + 18)x + 9n + 12.
\]
For  $x\ge n+1$, we have
\[
g_{16}'(x)\ge g_{16}'(n+1)=n^2 - 5n - 3>0.
\]
Then, for $x\ge n+2-\frac{4}{n+2}$,
\[
g_{16}(x)\ge g_{16}\left(n+2-\frac{4}{n+2}\right)=\frac{2n^5 - 8n^4 - 62n^3 + 36n^2 + 360n + 208}{n^3 + 6n^2 + 12n + 8}>0.
\]
From Lemmas and \ref{quo} and \ref{712}, we have $q(G)<n+2-\frac{4}{n+2}<q(K_{1,1,n-2}^+)$.

\noindent
4. Proof of $q(K_1\vee (K_{1,s}^+\cup \frac{n-s-2}{4}K_4))<q(K_{1,1,n-2}^+)$ where $2\le s\le n-6$:

Let $G= K_1\vee (K_{1,s}^+\cup \frac{n-s-2}{4}K_4)$.
If $s=2$, then
$Q(G)$ has an equitable quotient matrix
$
B_{17}=
\left(\begin{smallmatrix}
n-1 &   3 & n-4   \\
1  &  5 & 0\\
1 &  0 & 7
\end{smallmatrix}\right)$
with CP
\[
g_{17}(x)=x^3 - ( n + 11)x^2 + (11n + 24)x - 30n + 36.
\]
Then
\[
g_{17}'(x)=3x^2-(2n+22)x+11n+24
\]
For $x\ge n+1$, as $n\ge 8$, we have
\[
g_{17}'(x)\ge g_{17}'(n+1)=n^2 - 7n + 5>0.
\]
Then, for $x\ge n+2-\frac{4}{n+2}$,
\[
g_{17}(x)\ge g_{17}\left(n+2-\frac{4}{n+2}\right)=\frac{2n^5 - 12n^4 - 52n^3 + 160n^2 + 576n + 288}{n^3 + 6n^2 + 12n + 8}>0.
\]
From Lemmas \ref{quo} and \ref{712}, we have $q(G)<n+2-\frac{4}{n+2}<q(K_{1,1,n-2}^+)$.
Suppose now that $s\ge 3$.
Then
$Q(G)$ has an equitable quotient matrix
$B_{18}=
\left(\begin{smallmatrix}
n-1 & 1 &  2 &  s-2 & n-s-2   \\
1  & s+1 & 2 & s-2 & 0\\
1  & 1 & 4 & 0 & 0\\
1 &  1 & 0 & 2 & 0\\
1 & 0 & 0 & 0 &7
\end{smallmatrix}\right)$
with CP
\begin{align*}
g_{18}(x,s)&=x^5 - (n +s + 13)x^4 + (13n + 11s + ns + 50)x^3 - (56n + 26s + 11ns + 52)x^2 \\
&\quad - ( 6s^2 - 34ns - 96n + 24)x + 24s^2 - 48s - 24ns - 72n + 144.
\end{align*}
Let $h_2(x)=g_{18}(x,s)-g_{18}(x,s+4)$. Then
\begin{align*}
h_2(x)&=4x^4 -(4n+ 44)x^3 + (44n + 104)x^2 - (136n+48s + 96)x + 96n- 192s - 192,\\
h_2'(x)&=16x^3- (12n + 132)x^2 +(88n + 208)x - 136n- 48s - 96,\\
h_2''(x)&=48x^2 - ( 24n + 264)x + 88n + 208.
\end{align*}
For  $x\ge n+1$, we have
\begin{align*}
h_2''(x)\ge h_2''(n+1)&=24n^2 - 104n - 8>0,\\
h_2'(x)\ge h_2'(n+1)&=4n^3 - 20n^2 - 68n - 48s - 4\\
&\ge 4n^3 - 20n^2 - 116n + 284>0,
\end{align*}
so
\begin{align*}
h_2(x)\ge h_2(n+1)&=4n^3 - 64n^2 - (48s + 4)n - 240s - 224\\
&\ge 4n^3 - 112n^2 + 44n + 1216>0.
\end{align*}
From Lemma \ref{quo}, $q(G(n,s))<q(G(n,s+4))$.
Note that $G(n,n-2)\cong K_{1,1,n-2}^+$.
So $q(G)< q(K_{1,1,n-2}^+)$.
\end{appendices}
\end{document}